\documentclass[a4paper,11pt]{amsart}
\usepackage{mathrsfs}
\usepackage{cases}
\usepackage{multirow}
\usepackage{color}
\usepackage{colordvi}
\usepackage{amsfonts}

\usepackage{graphicx}
\usepackage{amsmath}
\usepackage{amssymb}
\usepackage[all]{xypic}
\usepackage[all]{xy}

\begin{document}
\input xy
\xyoption{all}

\renewcommand{\mod}{\operatorname{mod}\nolimits}
\newcommand{\proj}{\operatorname{proj}\nolimits}
\newcommand{\rad}{\operatorname{rad}\nolimits}
\newcommand{\Gproj}{\operatorname{Gproj}\nolimits}
\newcommand{\Ginj}{\operatorname{Ginj}\nolimits}
\newcommand{\Gd}{\operatorname{Gd}\nolimits}
\newcommand{\soc}{\operatorname{soc}\nolimits}
\newcommand{\ind}{\operatorname{inj.dim}\nolimits}
\newcommand{\Top}{\operatorname{top}\nolimits}
\newcommand{\ann}{\operatorname{Ann}\nolimits}
\newcommand{\id}{\operatorname{id}\nolimits}
\newcommand{\Id}{\operatorname{id}\nolimits}
\newcommand{\irr}{\operatorname{irr}\nolimits}
\newcommand{\Mod}{\operatorname{Mod}\nolimits}
\newcommand{\End}{\operatorname{End}\nolimits}
\newcommand{\Ob}{\operatorname{Ob}\nolimits}
\newcommand{\Aus}{\operatorname{Aus}\nolimits}
\newcommand{\ver}{\operatorname{v}\nolimits}
\newcommand{\arr}{\operatorname{a}\nolimits}
\newcommand{\cone}{\operatorname{cone}\nolimits}
\newcommand{\rep}{\operatorname{rep}\nolimits}
\newcommand{\Ext}{\operatorname{Ext}\nolimits}
\newcommand{\Hom}{\operatorname{Hom}\nolimits}
\newcommand{\RHom}{\operatorname{RHom}\nolimits}
\renewcommand{\Im}{\operatorname{Im}\nolimits}
\newcommand{\Ker}{\operatorname{Ker}\nolimits}
\newcommand{\Coker}{\operatorname{Coker}\nolimits}
\renewcommand{\dim}{\operatorname{dim}\nolimits}
\newcommand{\Ab}{{\operatorname{Ab}\nolimits}}
\newcommand{\Coim}{{\operatorname{Coim}\nolimits}}
\newcommand{\pd}{\operatorname{proj.dim}\nolimits}
\newcommand{\Ind}{\operatorname{Ind}\nolimits}
\newcommand{\add}{\operatorname{add}\nolimits}
\newcommand{\pr}{\operatorname{pr}\nolimits}
\newcommand{\Tr}{\operatorname{Tr}\nolimits}
\newcommand{\Def}{\operatorname{Def}\nolimits}
\newcommand{\Gp}{\operatorname{Gproj}\nolimits}

\newcommand{\ca}{{\mathcal A}}
\newcommand{\cb}{{\mathcal B}}
\newcommand{\cc}{{\mathcal C}}
\newcommand{\cd}{{\mathcal D}}
\newcommand{\cg}{{\mathcal G}}
\newcommand{\cp}{{\mathcal P}}
\newcommand{\ce}{{\mathcal E}}
\newcommand{\cs}{{\mathcal S}}
\newcommand{\cm}{{\mathcal M}}
\newcommand{\cn}{{\mathcal N}}
\newcommand{\cx}{{\mathcal X}}
\newcommand{\ct}{{\mathcal T}}
\newcommand{\cu}{{\mathcal U}}
\newcommand{\co}{{\mathcal O}}
\newcommand{\cv}{{\mathcal V}}
\newcommand{\calr}{{\mathcal R}}
\newcommand{\ol}{\overline}
\newcommand{\ul}{\underline}
\newcommand{\st}{[1]}
\newcommand{\ow}{\widetilde}
\newcommand{\coh}{{\mathrm coh}}
\newcommand{\CM}{{\mathrm CM}}
\newcommand{\vect}{{\mathrm vect}}

\newcommand{\bp}{{\mathbf p}}
\newcommand{\bL}{{\mathbf L}}
\newcommand{\bS}{{\mathbf S}}

\newtheorem{theorem}{Theorem}[section]
\newtheorem{acknowledgement}[theorem]{Acknowledgement}
\newtheorem{algorithm}[theorem]{Algorithm}
\newtheorem{axiom}[theorem]{Axiom}
\newtheorem{case}[theorem]{Case}
\newtheorem{claim}[theorem]{Claim}
\newtheorem{conclusion}[theorem]{Conclusion}
\newtheorem{condition}[theorem]{Condition}
\newtheorem{conjecture}[theorem]{Conjecture}
\newtheorem{construction}[theorem]{Construction}
\newtheorem{corollary}[theorem]{Corollary}
\newtheorem{criterion}[theorem]{Criterion}
\newtheorem{definition}[theorem]{Definition}
\newtheorem{example}[theorem]{Example}
\newtheorem{exercise}[theorem]{Exercise}
\newtheorem{lemma}[theorem]{Lemma}
\newtheorem{notation}[theorem]{Notation}
\newtheorem{problem}[theorem]{Problem}
\newtheorem{proposition}[theorem]{Proposition}
\newtheorem{remark}[theorem]{Remark}
\newtheorem{solution}[theorem]{Solution}
\newtheorem{summary}[theorem]{Summary}
\newtheorem*{thm}{Theorem}
\newtheorem*{thma}{Theorem A}
\newtheorem*{thmb}{Theorem B}
\newtheorem*{thmc}{Theorem C}
\newtheorem*{thm1}{Main Theorem 1}
\newtheorem*{thm2}{Main Theorem 2}

\def \A{{\Bbb A}}
\def \Z{{\Bbb Z}}
\def \X{{\Bbb X}}
\renewcommand{\L}{{\Bbb L}}
\renewcommand{\P}{{\Bbb P}}
\newcommand {\lu}[1]{\textcolor{red}{$\clubsuit$: #1}}
\title[Gorenstein properties of simple gluing algebras]{Gorenstein properties of simple gluing algebras}

\author[Lu]{Ming Lu}
\address{Department of Mathematics, Sichuan University, Chengdu 610064, P.R.China}
\email{luming@scu.edu.cn}


\begin{abstract}
Let $A=KQ_A/I_A$ and $B=KQ_B/I_B$ be two finite-dimensional bound quiver algebras, fix two vertices $a\in Q_A$ and $b\in Q_B$.
We define an algebra $\Lambda=KQ_\Lambda/I_\Lambda$, which is called a simple gluing algebra of $A$ and $B$, where $Q_\Lambda$ is from $Q_A$ and $Q_B$ by identifying $a$ and $b$, $I_\Lambda=\langle I_A,I_B\rangle$. We prove that $\Lambda$ is Gorenstein if and only if $A$ and $B$ are Gorenstein, and describe the Gorenstein projective modules, singularity category, Gorenstein defect category and also Cohen-Macaulay Auslander algebra of $\Lambda$ from the corresponding ones of
$A$ and $B$.
\end{abstract}

\maketitle

\section{Introduction}

In the study of B-branes on Landau-Ginzburg models in the framework of Homological Mirror Symmetry Conjecture, D. Orlov rediscovered the notion of singularity categories \cite{Or1,Or2,Or3}. The singularity category of an algebra $A$ is defined to be the Verdier quotient of the bounded derived category with respect to the thick subcategory formed by complexes isomorphic to those consisting of finitely generated projective modules, \cite{Bu}. It measures the homological singularity of an algebra in the sense that an algebra $A$ has finite global dimension if and only if its singularity category vanishes.

The singularity category captures the stable homological features of an algebra \cite{Bu}. A fundamental result of R. Buchweitz \cite{Bu} and D. Happel \cite{Ha1} states that for a Gorenstein algebra $A$, the singularity category is triangle equivalent to the stable category of Gorenstein projective (also called (maximal) Cohen-Macaulay) $A$-modules. Buchweitz's Theorem (\cite[Theorem 4.4.1]{Bu}) says that there is an exact embedding $\Phi:\underline{\Gp}A\rightarrow D_{sg}(A)$ given by $\Phi(M)=M$, where the second $M$ is the corresponding stalk complex at degree $0$, and $\Phi$ is an equivalence if and only if $A$ is Gorenstein. Recently, to provide a categorical characterization of Gorenstein algebras, P. A. Bergh, D. A. J{\o}rgensen and S. Oppermann \cite{BJO} defined the Gorenstein defect category $D_{def}(A):=D_{sg}(A)/\Im \Phi$ and proved that $A$ is Gorenstein if and only if $D_{def}(A)=0$. In general, it is difficult to describe the singularity categories and Gorenstein defect categories. Many people are trying to describe these categories for some special kinds of algebras, see e.g. \cite{Chen1,chen2,chen3,Ka,CGLu,CDZ}. In particular, for a CM-finite algebra $A$, F. Kong and P. Zhang \cite{KZ} proved that its Gorenstein defect category is equivalent to the singularity category of its Cohen-Macaulay Auslander algebra. Recently, X-W. Chen \cite{chen3} described the singularity category and Gorenstein defect category for a quadratic monomial algebra.
For a triangular matrix algebra $\Lambda=\left( \begin{array}{cc} A&M\\0&B \end{array}\right)$, the relations between these categories of $\Lambda$ and the corresponding ones of $A$ and $B$ are decribed clearly in some sense \cite{Chen1,KZ,XZ,Zhang2,Lu}.

Inspired by the above, we define a new algebra (called \emph{simple gluing algebra}) from some algebras. Explicitly, let $A=KQ_A/I_A$, $B=KQ_B/I_B$ be two finite-dimensional algebras. For any two vertices $a\in Q_A$, $b\in Q_B$, we define a new quiver $Q$ from $Q_A$ and $Q_B$ by identifying $a$ and $b$. In this way, we can view $Q_A$ and $Q_B$ as subquivers of $Q$. We call $Q$ the \emph{simple gluing quiver} of $Q_A$ and $Q_B$. Denote by $v\in Q$ the \emph{glued vertex}. Let $I$ be the ideal of $KQ$ generated by $I_A$ and $I_B$. Then $\Lambda=KQ/I$ is called a simple gluing algebra of $A$ and $B$ if $\Lambda$ is finite-dimensional.
Similar to the triangular matrix algebras, first we prove that $\Lambda$ is Gorenstein if and only if $A$ and $B$ are Gorenstein, see Proposition \ref{lemma simple gluing Nakayama algebra Gorenstein}; second we prove that $D_{sg}(\Lambda)\simeq D_{sg}(A) \coprod D_{sg}(B)$ (see Theorem \ref{theorem singularity categories}), $\underline{\Gp}(\Lambda)\simeq \underline{\Gp}(A) \coprod \underline{\Gp}(B)$ (see Theorem \ref{theorem stable category of Cm modules}) and $D_{def}(\Lambda)\simeq D_{def}(A) \coprod D_{def}(B)$ (see Corollary \ref{corollary Gorenstein defect categories}). In particular, if we know the Gorenstein projective modules over $A$ and $B$, then we can get all the Gorenstein projective modules over $\Lambda$. Finally, we
prove that the Cohen-Macaulay Auslander algebra of $\Lambda$ is a simple gluing algebra of the Cohen-Macaulay Auslander algebras of $A$ and $B$, see Theorem \ref{theorem Cohen-Macaulay Auslander algebras}. As applications, we redescribe the singularity categories for cluster-tilted algebras of type $\A$ and endomorphism algebras of maximal rigid objects of cluster tube $\cc_n$.

\vspace{0.2cm} \noindent{\bf Acknowledgments.}
The work was done during the stay of the author at the Department of Mathematics,
University of Bielefeld. He is deeply indebted to Professor Henning Krause for his kind
hospitality, inspiration and continuous encouragement.
The author thanks Professor Liangang Peng very much for his guidance and constant support. The author was supported by the National Natural Science Foundation of China (No. 11401401 and No. 11601441).

\section{preliminary}
In this paper, we always assume that $K$ is an algebraically closed field and all algebras are finite-dimensional algebras over $K$ and modules are finitely generated.

Let $A$ be a $K$-algebra. Let $\mod A$ be the category of finitely generated left $A$-modules. With $D=\Hom_K(-,K)$ we denote the standard duality with respect to the ground field. Then $D(A_A)$ is an injective cogenerator for $\mod A$. For an arbitrary $A$-module $_AX$ we denote by $\pd_AX$ (resp. $\ind_AX$) the projective dimension (resp. the injective dimension) of the module $_AX$.

A complex $$P^\bullet:\cdots\rightarrow P^{-1}\rightarrow P^0\xrightarrow{d^0}P^1\rightarrow \cdots$$ of finitely generated projective $A$-modules is said to be \emph{totally acyclic} provided it is acyclic and the Hom complex $\Hom_A(P^\bullet,A)$ is also acyclic \cite{AM}.
An $A$-module $M$ is said to be (finitely generated) \emph{Gorenstein projective} provided that there is a totally acyclic complex $P^\bullet$ of projective $A$-modules such that $M\cong \Ker d^0$ \cite{EJ}. We denote by $\Gproj A$ the full subcategory of $\mod A$ consisting of Gorenstein projective modules.

An algebra is of \emph{finite Cohen-Macaulay type}, or simply, \emph{CM-finite}, if there are only finitely many isomorphism classes of indecomposable finitely generated Gorenstein projecitve modules. Clearly, $A$ is CM-finite if and only if there is a finitely generated module $E$ such that $\Gproj A=\add E$. In this way, $E$ is called to be a \emph{Gorenstein projective generator}. If $A$ is self-injective, then $\Gproj A=\mod A$, so $A$ is CM-finite if and only if $A$ is representation-finite. An algebra is called \emph{CM-free} if $\Gproj A=\proj A$. If $A$ has finite global dimension, then $\Gproj A=\proj A$, so it is CM-free.

Let $A$ be a CM-finite algebra, $E_1,\dots,E_n$ all the pairwise non-isomorphic indecomposable Gorenstein projective $A$-modules. Put $E=\oplus_{i=1}^n E_i$. Then $E$ is a Gorenstein projective generator. We call $\Aus(\Gproj A):=(\End_A E)^{op}$ the \emph{Cohen-Macaulay Auslander algebra}(also called \emph{relative Auslander algebra}) of $A$.

Let $\cx$ be a subcategory of $\mod A$. Then $^\bot\cx:=\{M|\Ext^i(M,X)=0, \mbox{ for all } X\in\cx, i\geq1\}$. Dually, we can define $\cx^\bot$. In particular, we define $^\bot A:=\,^\bot
(\proj A)$.

The following lemma follows from the definition of Gorenstein projective module easily.
\begin{lemma}
(i) \cite{Be}
\begin{eqnarray*}
\Gp(A)&=&\{M\in \Lambda\mbox{-}\mathrm{mod}\,|\,\exists \mbox{ an exact sequence }
0\rightarrow M\rightarrow T^0\xrightarrow{d^0}T^1\xrightarrow{d^1}\cdots, \\ &&\mbox{ with }T^i\in\proj A,\ker d^i\in\,^\bot A,\forall i\geq0\}.
\end{eqnarray*}

(ii) If $M$ is Gorenstein projective, then $\Ext^i_A(M,L)=0$, $\forall i>0$, for all $L$ of finite projective dimension or of finite injective dimension.

(iii) If $P^\bullet$ is a totally acyclic complex, then all $\Im d^i$ are Gorenstein projective; and any
truncations
$$\cdots\rightarrow P^i\rightarrow\Im d^i\rightarrow0,\quad 0\rightarrow\Im d^i\rightarrow P^{i+1}\rightarrow\cdots$$
and
$$0\rightarrow\Im d^i\rightarrow P^{i+1}\rightarrow\cdots\rightarrow P^j\rightarrow \Im d^j\rightarrow0,i<j$$
are $\Hom_A(-,\proj A)$-exact.
\end{lemma}

\begin{definition}[\cite{Ha1}, see also \cite{AR1,AR2}]
A finite-dimensional algebra $A$ is called a Gorenstein algebra (also called Iwanaga-Gorenstein algebra) if $A$ satisfies $\ind A_A<\infty$ and $\ind_AA<\infty$.
Given an $A$-module $X$. If $\Ext^i_A(X,A)=0$ for all $i>0$, then $X$ is called a Cohen-Macaulay module of $A$.
\end{definition}

Observe that for a Gorenstein algebra $A$, we have $\ind _AA=\ind A_A$, see \cite[Lemma 6.9]{Ha1}; the common value is denoted by $\Gd A$. If $\Gd A\leq d$, we say that $A$ is \emph{$d$-Gorenstein}. Furthermore, since $\pd _A D(A_A)=\ind A_A$, we get that $A$ is Gorenstein if and only if $\pd _A D(A_A)<\infty$ and $\ind_AA<\infty$.

\begin{theorem}[\cite{Bu,EJ}]
Let $A$ be a Gorenstein algebra. Then

(i) If $P^\bullet$ is an exact sequence of projective left $A$-modules, then $\Hom_A(P^\bullet,A)$ is again an exact sequence of projective right $A$-modules.

(ii) A module $G$ is Gorenstein projective if and only if there is an exact sequence $0\rightarrow G\rightarrow P^0\rightarrow P^1\rightarrow \cdots$ with each
$P^i$ projective.

(iii) $\Gproj A=\,^\bot A$.
\end{theorem}
For a module $M$, take a short exact sequence $$0\rightarrow \Omega M\rightarrow P\rightarrow M\rightarrow0$$
with $P$ projective. The module $\Omega M$ is called a \emph{syzygy module} of $M$. For each $i\geq1$, denote by $\Omega^i$ the $i$-th power of $\Omega$ and then for a module $X$, $\Omega^i X$ is the $i$-th syzygy module of $X$. For details, see \cite{ARS}.
\begin{theorem}[\cite{AM}]
Let $A$ be an finite-dimensional algebra and $d\geq0$. Then the following statements are equivalent:

(i) the algebra $A$ is $d$-Gorenstein;

(ii) $\Gproj A=\Omega^d(\mod A)$.

\noindent In this case, we have $\Gproj A=\, ^\bot A$.
\end{theorem}
So for a Gorenstein algebra, the definition of Cohen-Macaulay module coincides with the one of Gorenstein projective.

Recall that for an algebra $A$, the \emph{singularity category} of $A$ is the quotient category $D_{sg}(A):=D^b(A)/K^b(\proj A)$, which is defined by Buchweitz \cite{Bu}, see also \cite{Ha1,Or1}.

\begin{theorem}[Buchweitz's Theorem, see also \cite{KV} for a more general version]\label{theorem stable category of CM modules }
Let $A$ be an Artin algebra. Then $\Gproj (A)$ is a Frobenius category with the projective modules as the projective-injective objects, and there is an exact embedding $\Phi:\underline{\Gp}A\rightarrow D_{sg}(A)$ given by $\Phi(M)=M$, where the second $M$ is the corresponding stalk complex at degree $0$, and $\Phi$ is an equivalence if and only if $A$ is Gorenstein.
\end{theorem}

Let $A$ be an Artin algebra. Inspired by Buchweitz's Theorem, the \emph{Gorenstein defect category} is defined to be Verdier quotient $D_{def}(A):=D_{sg}(A)/\Im(\Phi)$, see \cite{BJO}.
From \cite{KZ}, we know that $D_{def}(A)$ is triangle equivalent to $D^b(A)/\langle \Gp(A)\rangle$, where $\langle \Gp(A)\rangle$ denotes the triangulated subcategory of $D^b(A)$ generated by $\Gp(A)$, i.e., the smallest triangulated subcategory of $D^b(A)$ containing $\Gp(A)$.

\begin{lemma}[\cite{BJO,KZ}]
Let $A$ be an Artin algebra. Then the following are equivalent.

(i) $A$ is Gorenstein;

(ii) $\underline{\Gp}(A)$ is triangle equivalent to $D_{sg}(A)$;

(iii) $D_{def}(A)=0$;

(iv) $D^b(A)=\langle \Gp(A)\rangle$.
\end{lemma}

\section{Simple gluing algebras}

Let $Q=(\ver(Q),\arr(Q))$ be a finite quiver, where $\ver(Q)$ is the set of vertices, and $\arr(Q)$ is the set of arrows. For any arrow $\alpha$ in $Q$, we denote by $s(\alpha),t(\alpha)$ the source and the target of $\alpha$ respectively. The path algebra $KQ$ of $Q$ is an associative algebra with an identity. If $I$ is an admissible ideal of $KQ$, the pair $(Q,I)$ is said to be a bound quiver. The quotient algebra $KQ/I$ is said to be the algebra of the bound quiver $(Q,I)$, or simply, a \emph{bound quiver algebra}. Denote by $e_i$ the idempotent corresponding to the vertex $i\in Q$.

In this paper, we shall freely identify the representations of $(Q,I)$ with left modules over $KQ/I$.
Recall that a representation $M$ of $(Q,I)$ is of form $(M_i,M_\alpha)_{i\in \ver(Q),\alpha\in\arr(Q)}$, or $(M_i,M_\alpha)$. So for any representation $M$, we always use $M_i$ to denote the vector space associated to the vertex $i$.

As the beginning, let us recall the definition of \emph{simple gluing algebras}.
Let $A=KQ_A/I_A$, $B=KQ_B/I_B$ be two finite-dimensional algebras. For any two vertices $a\in \ver(Q_A)$, $b\in \ver(Q_B)$, we define a new quiver $Q$ from $Q_A$ and $Q_B$ by identifying $a$ and $b$. In this way, we can view $Q_A$ and $Q_B$ as subquivers of $Q$. We call $Q$ the \emph{simple gluing quiver} of $Q_A$ and $Q_B$. Denote by $v\in \ver(Q)$ the \emph{glued vertex}. Let $I$ be the ideal of $KQ$ generated by $I_A$ and $I_B$.

The following lemma follows directly from the definition of bound quiver algebras.
\begin{lemma}
Keep the notations as above. Then $\Lambda=KQ/I$ is a finite-dimensional algebra if and only if each non-trivial path in $Q_A$ from $a$ to $a$ is in $I_A$ or each non-trivial path in $Q_B$ from $b$ to $b$ is in $I_B$.
\end{lemma}

In the following, we always assume that the above condition holds, and call $\Lambda$ the \emph{simple gluing algebra} of $A$ and $B$. Inductively, from finitely many bound quiver algebras $A_1,\dots,A_n$, we can define simple gluing algebra (which is finite-dimensional) of $A_1,\dots,A_n$.

Without loss of generality, we always assume that each non-trivial path in $Q_A$ from $a$ to $a$ is in $I_A$. Denote by $e_B=\sum_{i\in Q_B} e_i$ the idempotent of $\Lambda$. Obviously, $A$ and $B$ are subalgebras (and also quotient algebras) of $\Lambda$, and
$B\cong e_B \Lambda e_B$.

Since $B=e_B\Lambda e_B$ is a subalgebra of $\Lambda$, there exists an exact functor $j_\mu:\mod \Lambda\rightarrow \mod B$, which maps $M$ to $e_B M$. Note that $j_\mu=e_B\Lambda\otimes_\Lambda-$. In fact, $j_\mu$ restricts the representation of $(Q,I)$ to the representation of $(Q_B,I_B)$. $j_\mu$ admits a right adjoint functor $j_\rho:=\Hom_B(_B\Lambda,-):\mod B\rightarrow \mod\Lambda$, and admits a left adjoint functor $j_\lambda:=\Lambda\otimes_B-=\Lambda e_B\otimes_B-: \mod B\rightarrow \mod\Lambda$, see e.g. \cite{CPS2}.

Similarly, since $A$ is also a subalgebra of $\Lambda$, there exists a natural exact functor $i_\mu:\mod \Lambda\rightarrow \mod A$.
$i_\mu$ admits a right adjoint functor $i_\rho:=\Hom_A(_A\Lambda,-):\mod A\rightarrow \mod \Lambda$, and admits a left adjoint functor $i_\lambda:= \Lambda\otimes_A-: \mod A\rightarrow \mod \Lambda$.

From the structure of $\Lambda$, obviously, $\Lambda$ is projective both as left (also right) modules over $A$ and $B$. So we have the following lemma immediately.

\begin{lemma}
Keep the notations as above. Then all the functors $j_\lambda,j_\mu,j_\rho$, $i_\lambda,i_\mu,i_\rho$ are exact functors. In particular, $i_\lambda$, $i_\mu$, $j_\mu$ and $j_\lambda$ preserve projectives, $i_\rho$, $i_\mu$, $j_\mu$ and $j_\rho$ preserve injectives.
\end{lemma}

For any vertex $i\in \ver(Q_A)$, we denote by $P(i)$ (resp. $_A S(i)$, $I(i)$) the projective (resp. simple, injective) $A$-module corresponding to $i$.
For any vertex $i\in \ver(Q_B)$, we denote by $Q(i)$ (resp. $_B S(i)$, $J(i)$) the projective (resp. simple, injective) $B$-module corresponding to $i$.
For any vertex $i\in \ver(Q)$, we denote by $U(i)$ (resp. $_\Lambda S(i)$, $V(i)$) the projective (resp. simple, injective) $\Lambda$-module corresponding to $i$.

Note that $i_\lambda(P(i))=U(i)$, $i_\rho(I(i))=V(i)$ for any $i\in Q_A$, and $j_\lambda(Q(i))= U(i)$ and $j_\rho(J(i))=V(i)$ for any $i\in Q_B$.

In the following, we describe the actions of all these functors on representations.
First, from the structure of indecomposable projective modules \cite{ASS}, we can assume that $Q(v)$ is locally as the following left figure shows, where $j_1,\dots j_q\in Q_A$, $k_1,\dots,k_r\in Q_B$ and each of them is the ending point of some arrow starting at $v$. Then $j_\lambda(\,_B S(b))$ and $i_\lambda(\,_AS(a))$ are as the following middle and right figures show respectively.
\begin{center}\setlength{\unitlength}{0.7mm}
 \begin{picture}(70,30)(0,10)
\put(-40,10){\begin{picture}(50,10)
\put(20,20){$v$}

\put(0,0){$j_1$}

\put(7,0){$\cdots$}
\put(15,0){$j_q$}

\put(25,0){$k_1$}

\put(32,0){$\cdots$}
\put(40,0){$k_r$}

\put(18,18){\vector(-1,-1){14}}
\put(20.5,18){\vector(-1,-4){3.3}}

\put(21.5,18){\vector(1,-4){3.3}}

\put(23,18){\vector(1,-1){14}}

\put(1,-8){$\vdots$}

\put(16,-8){$\vdots$}

\put(26,-8){$\vdots$}
\put(41,-8){$\vdots$}

\end{picture}}

 \put(30,10){\begin{picture}(50,10)

\put(20,20){$v$}

\put(0,0){$j_1$}

\put(7,0){$\cdots$}
\put(15,0){$j_q$}

\put(18,18){\vector(-1,-1){14}}
\put(20.5,18){\vector(-1,-4){3.3}}

\put(1,-8){$\vdots$}

\put(16,-8){$\vdots$}

\end{picture}}

\put(60,10){\begin{picture}(50,10)
\put(20,20){$v$}

\put(25,0){$k_1$}

\put(32,0){$\cdots$}
\put(40,0){$k_r$}

\put(21.5,18){\vector(1,-4){3.3}}

\put(23,18){\vector(1,-1){14}}

\put(26,-8){$\vdots$}
\put(41,-8){$\vdots$}

\end{picture}}

\put(-40,-10){Figure 1. The structures of $Q(v)$, $j_\lambda(\,_BS(b))$ and $i_\lambda(\,_AS(a))$.}
\end{picture}

\vspace{1.9cm}
\end{center}

Let $M=(M_i,M_\alpha)$ be a representation of $(Q_A,I_A)$. We assume that $M$ is locally as the following left figure shows, where $i_1,\dots i_p$ are the starting points of arrows ending to $v=a$ in $Q_A$. Then $i_\lambda(M)$ is as the following right figure shows, where the submodule in the dashed box is $i_\lambda(\,_AS(a))^{\oplus\dim M_v }$.

\begin{center}\setlength{\unitlength}{0.7mm}
 \begin{picture}(80,55)(0,10)

 \put(-30,10){\begin{picture}(50,10)

\put(-5,40){$M_{i_1}$}

\put(5,40){$\cdots$}
\put(15,40){$M_{i_p}$}

\put(18,20){$M_v$}

\put(-5,0){$M_{j_1}$}

\put(5,0){$\cdots$}
\put(15,0){$M_{j_q}$}

\put(3,38){\vector(1,-1){14}}
\put(17,38){\vector(1,-4){3.2}}
\put(17,18){\vector(-1,-1){14}}
\put(20.5,18){\vector(-1,-4){3.3}}

\put(-3,46){$\vdots$}

\put(17,46){$\vdots$}

\put(-3,-8){$\vdots$}

\put(17,-8){$\vdots$}

\end{picture}}

\put(50,10){\begin{picture}(50,10)
\put(-5,40){$M_{i_1}$}

\put(5,40){$\cdots$}
\put(15,40){$M_{i_p}$}

\put(18,20){$M_v$}

\put(-5,0){$M_{j_1}$}

\put(5,0){$\cdots$}
\put(15,0){$M_{j_q}$}

\put(3,38){\vector(1,-1){14}}
\put(17,38){\vector(1,-4){3.2}}
\put(17,18){\vector(-1,-1){14}}
\put(20.5,18){\vector(-1,-4){3.3}}

\put(22,18){\vector(1,-2){6.3}}
\put(23,18){\vector(2,-1){24}}

\put(44,0){$M_{k_r}$}
\put(36,0){$\cdots$}
\put(27,0){$M_{k_1}$}

\put(-3,46){$\vdots$}

\put(17,46){$\vdots$}

\put(-3,-8){$\vdots$}

\put(17,-8){$\vdots$}

\put(46,-8){$\vdots$}

\put(29,-8){$\vdots$}

\qbezier[40](17,24)(21,10)(28,-9)
\qbezier[40](28,-9)(40,-12)(60,-9)
\qbezier[40](17,24)(30,26)(50,24)
\qbezier[40](50,24)(55,15)(60,-9)

\put(35,27){$i_\lambda(\,_AS(a))^{\oplus\dim M_v }$}
\end{picture}}

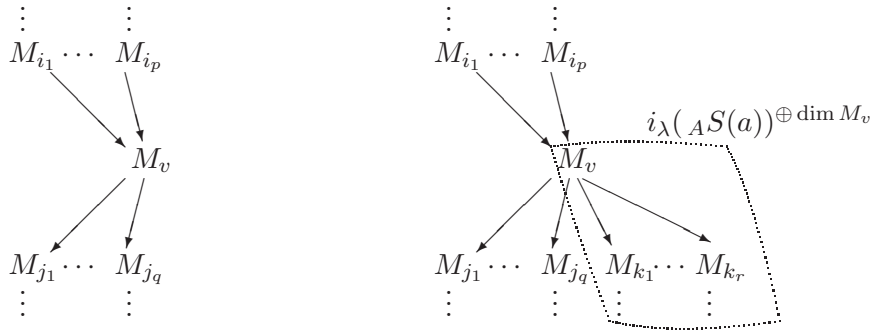
\put(-40,-10){Figure 2. The local structures of $M\in \mod A$ and $i_\lambda(M)$.}
\end{picture}

\vspace{1.9cm}
\end{center}

The action of $j_\lambda$ is similar, and $i_\mu$, $j_\mu$ are the restriction functors. For $i_\rho$ and $j_\rho$, their actions are dual to these of $i_\lambda$ and $j_\lambda$.

For any $M\in\mod A$ and $N\in\mod \Lambda$, we denote by
$$\alpha_{M,N}: \Hom_\Lambda(i_\lambda(M),N)\xrightarrow{\sim} \Hom_A(M,i_\mu(N))$$
the adjoint isomorphism.
For any $M\in \mod A$, we denote by $\mu_M= \alpha_{M, i_\lambda (M)} (1_{i_\lambda (M)})$ the adjunction morphism, and for any $N\in \mod \Lambda$, we denote by
$\epsilon_N= \alpha_{i_\mu(N),N}^{-1}(1_{i_\mu (N)})$.

Similarly, for any $L\in\mod B$ and $N\in\mod \Lambda$ we denote by
$$\beta_{L,N}: \Hom_\Lambda(j_\lambda(L),N)\xrightarrow{\sim} \Hom_B(L,i_\mu(N))$$
the adjoint isomorphism.
For any $L\in \mod B$, we denote by $\nu_L= \beta_{L, j_\lambda (L)} (1_{j_\lambda (L)})$ the adjunction morphism, and for any $N\in \mod \Lambda$, we denote by
$\zeta_N= \beta_{j_\mu(N),N}^{-1}(1_{j_\mu (N)})$.

It is easy to see that for any $M\in \mod A$, $i_\mu i_\lambda M= M\oplus P(a)^{\oplus m}$ for some $m$, and $\mu_M$ is the section map.
Furthermore, for any morphism
$$f=(f_x)_{x\in \ver(Q_A)} : M=(M_x,M_\alpha)_{x\in \ver( Q_A),\alpha\in \arr(Q_A)}\rightarrow  N=(N_x,N_\alpha)_{x\in \ver( Q_A),\alpha\in \arr(Q_A)}$$ in $\mod A$, we get that
$i_\mu i_\lambda (f)$ is of form
\begin{equation}\label{equation form of morphism 1}
 \left( \begin{array}{cc} f& 0\\
0& h \end{array} \right):  M\oplus P(a)^{\oplus m} \rightarrow  N\oplus P(a)^{\oplus n},
\end{equation}
where $h:P(a)^{\oplus m} \rightarrow P(a)^{\oplus n}$ can be represented by a $n\times m$ matrix with its entries $h_{ij}:P(a)\rightarrow P(a)$ in $K$. In particular, $h$ is determined by the map $f_v$.

For any morphism
$$g=(g_y)_{y\in \ver(Q_B)} : M=(M_y,M_\beta)_{y\in \ver( Q_B),\beta\in \arr(Q_B)}\rightarrow  N=(N_y,N_\beta)_{y\in \ver( Q_B),\beta\in \arr(Q_B)}$$ in $\mod B$, it is similar to get the form of $j_\mu j_\lambda (g)$.
Besides, we also get that $i_\mu j_\lambda (M)= P(a)^{\oplus m}$ for some $m$, and
$i_\mu j_\lambda (N)= P(a)^{\oplus n}$ for some $n$. Then $i_\mu j_\lambda(g)$ is of form
\begin{equation}\label{equation form of morphism 2}
h:  P(a)^{\oplus m} \rightarrow P(a)^{\oplus n},
\end{equation}
where $h:P(a)^{\oplus m} \rightarrow P(a)^{\oplus n}$ can be represented by a $n\times m$ matrix with its entries $h_{ij}:P(a)\rightarrow P(a)$ in $K$. In particular, $h$ is determined by the map $g_v$.

Similarly, we can describe $j_\mu i_\lambda (f)$ for any $f:M\rightarrow N$ in $\mod A$.

For any $N\in\mod \Lambda$, we denote by $l(N)$ its \emph{length}.

\begin{lemma}\label{lemma existence of short exact sequene 1}
Keep the notations as above. Let $M=(M_i,M_\gamma)$ be a $\Lambda$-module. Then we have a short exact sequence
\begin{equation}\label{equation short exact sequence 1}
0\rightarrow U(v)^{\oplus\dim M_v} \rightarrow i_\lambda i_\mu M\oplus j_\lambda j_\mu M\xrightarrow{(\mu_M,\nu_M)} M\rightarrow0.
\end{equation}
\end{lemma}
\begin{proof}
First, suppose that $M=\,_\Lambda S(i)$ is a simple $\Lambda$-module. Without loss of generality, we assume that $i\in \ver(Q_A)$.
If $i\neq v$, then $i_\lambda i_\mu (_\Lambda S(i))=\,_\Lambda S(i)$ and $j_\lambda j_\mu (_\Lambda S(i))=0$, which satisfies (\ref{equation short exact sequence 1}). If $i=v$, then $i_\mu(_\Lambda S(v))=\,_A S(a)$,
and $j_\mu(_\Lambda S(v))=\,_B S(a)$.
From the structures of $Q(v)$, $j_\lambda(\,_BS(b))$ and $i_\lambda(\,_AS(a))$ in Figure 1, it is easy to see that there exists a
short exact sequence
$$0\rightarrow U(v)\rightarrow i_\lambda i_\mu S(v)\oplus j_\lambda j_\mu S(v)\rightarrow S(v)\rightarrow0,$$
which satisfies the requirement.

For general $M$, we prove it by induction on its length. If $l(M)=1$, then $M$ is a simple module, and the result follows from the above.
If $l(M)>1$, then there exists $M_1,M_2$ such that $0<l(M_i)<l(M)$ for $i=1,2$, and there is a short exact sequence
$$0\rightarrow M_1\xrightarrow{f} M\xrightarrow{g} M_2\rightarrow0.$$
Obviously, $\dim M_v=\dim (M_1)_v+\dim (M_2)_v$ for the glued vertex $v\in \ver(Q)$.

Since $i_\lambda, i_\mu,j_\lambda,j_\mu$ are exact functors, there are two short exact sequences
$$0\rightarrow i_\lambda i_\mu M_1 \xrightarrow{i_\lambda i_\mu(f) } i_\lambda i_\mu M\xrightarrow{i_\lambda i_\mu(g) } i_\lambda i_\mu M_2\rightarrow0,$$
and
$$0\rightarrow j_\lambda j_\mu M_1 \xrightarrow{j_\lambda j_\mu(f) } j_\lambda j_\mu M\xrightarrow{j_\lambda j_\mu(g) } j_\lambda j_\mu M_2\rightarrow0.$$
The inductive assumption yields the following short exact sequences:
$$0\rightarrow U(v)^{\oplus\dim (M_1)_v} \rightarrow i_\lambda i_\mu M_1\oplus j_\lambda j_\mu M_1\xrightarrow{(\mu_{M_1},\nu_{M_1})} M_1\rightarrow0,$$
and
$$0\rightarrow U(v)^{\oplus\dim (M_2)_v} \rightarrow i_\lambda i_\mu M_2\oplus j_\lambda j_\mu M_2\xrightarrow{(\mu_{M_2},\nu_{M_2})} M_2\rightarrow0.$$

From the naturality of the adjoint pairs $(i_\lambda,i_\mu)$ and $(j_\lambda,j_\mu)$, we get the following commutative diagram
\[\xymatrix{ U(v)^{\oplus\dim (M_1)_v} \ar[r]  \ar@{.>}[dd] & i_\lambda i_\mu M_1\oplus j_\lambda j_\mu M_1 \ar[rr]^{\quad\quad(\mu_{M_1},\nu_{M_1})} \ar[dd]^{\tiny \left( \begin{array}{cc} i_\lambda i_\mu(f) &\\ &j_\lambda j_\mu(f) \end{array} \right)}&& M_1\ar[dd]^{f}\\
\\
N\ar[r] \ar@{.>}[dd] & i_\lambda i_\mu M\oplus j_\lambda j_\mu M \ar[rr]^{\quad\quad(\mu_{M},\nu_{M})} \ar[dd]^{ \tiny\left( \begin{array}{cc} i_\lambda i_\mu(g) &\\ &j_\lambda j_\mu(g) \end{array} \right)} &&M \ar[dd]^{g} \\
\\
 U(v)^{\oplus\dim (M_2)_v} \ar[r]  & i_\lambda i_\mu M_2\oplus j_\lambda j_\mu M_2 \ar[rr]^{\quad\quad(\mu_{M_2},\nu_{M_2})} &&M_2, }\]
where $N$ is the kernel of $(\mu_{M},\nu_{M})$. It is obvious that all the sequences appearing in the rows and columns of the above commutative diagram are short exact sequences. From the short exact sequence in the first column, we get that
$$N\cong U(v)^{\oplus \dim (M_1)_v+\dim (M_2)_v}= U(v)^{\oplus \dim M_v} $$
since $U(v)$ is projective, and then there is a short exact sequence
$$0\rightarrow U(v)^{\oplus\dim M_v} \rightarrow i_\lambda i_\mu M\oplus j_\lambda j_\mu M\xrightarrow{(\mu_M,\nu_M)} M\rightarrow0.$$
\end{proof}

Similar to Gorenstein property of the upper triangular matrix algebras obtained in \cite{Chen1}, we get the following result.
\begin{proposition}\label{lemma simple gluing Nakayama algebra Gorenstein}
Let $\Lambda$ be a simple gluing algebra of the two finite-dimensional bound quiver algebras $A$ and $B$. Then $\Lambda$ is Gorenstein if and only if $A$ and $B$ are Gorenstein. In particular, $\max\{\Gd A,\Gd B\}\leq\Gd (\Lambda)\leq \max\{\Gd A,\Gd B,1\}$. 
\end{proposition}
\begin{proof}
Keep the notations as above. If $\Lambda$ is Gorenstein, then for any indecomposable injective $\Lambda$-module $V(i)$, it has finite projective dimension. Let
$$0\rightarrow U_m \rightarrow \cdots U_1\rightarrow U_0\rightarrow V(i)\rightarrow0$$
be a projective resolution of $V(i)$. Apply $i_\mu$ to the above projective resolution. Since $i_\mu$ is exact and preserves projectives, we get that $\pd_A (i_\mu(V(i)))\leq m$.
It is easy to see that $i_\mu(V(i))= I(i)\oplus I(a)^{\oplus s}$ for some integer $s$. Then $\pd_A (I(i))\leq m$, and $\pd_A D(A)<\infty$. Similarly, we can get that $\ind_A A<\infty$ and so $A$ is Gorenstein. Furthermore, $\Gd A \leq \Gd \Lambda$.

Similarly, $B$ is Gorenstein with $\Gd B\leq \Gd\Lambda$. Then $\max\{\Gd A,\Gd B\}\leq \Gd\Lambda$.

Conversely, for any indecomposable injective $\Lambda$-module $V(i)$, $i_\mu (V(i))$ and $j_\mu (V(i))$ are injective as $A$-module and $B$-module respectively.
Since $A$ and $B$ are Gorenstein, $i_\mu (V(i))$ and $j_\mu(V(i))$ have finite projection dimensions. From $i_\lambda$ and $j_\lambda$ are exact and preserve projectives, we get that
$i_\lambda i_\mu(V(i))$ and $j_\lambda j_\mu(V(i))$ have finite projection dimensions. In particular,
\begin{eqnarray*}
&&\pd_\Lambda i_\lambda i_\mu (V(i)) \leq \pd_A i_\mu (V(i))\leq \Gd A \mbox{ and}\\
 &&\pd_\Lambda j_\lambda j_\mu(V(i))\leq \pd_B j_\mu(V(i))\leq \Gd B.
\end{eqnarray*}
So $\pd _\Lambda (i_\lambda i_\mu (V(i)) \oplus j_\lambda j_\mu(V(i)))\leq \max\{\Gd A,\Gd B\}$.

From the short exact sequence (\ref{equation short exact sequence 1}) in Lemma \ref{lemma existence of short exact sequene 1}, it is easy to see that
$$\pd_\Lambda V(i)\leq \max\{\Gd A,\Gd B,1\}.$$
So $\pd_\Lambda D\Lambda\leq \max\{\Gd A,\Gd B,1\}$.

Similarly,
we can get that $$\ind_\Lambda \Lambda \leq \max\{\Gd A,\Gd B,1\}.$$
So $\Lambda$ is Gorenstein, and $\max\{\Gd A,\Gd B\}\leq\Gd (\Lambda)\leq \max\{\Gd A,\Gd B,1\}$.
\end{proof}

From the above Proposition, if $A$ and $B$ are Gorenstein, then we have the following:
if $\Gd A=0=\Gd B$, then $\Gd(\Lambda)\leq 1$; otherwise, $\Gd (\Lambda)= \max\{\Gd A,\Gd B\}$.

\begin{example}
Let $A$ and $B$ be self-injective ($0$-Gorenstein) algebras, and $\Lambda$ be the simple gluing algebra of $A$ and $B$. In general, $\Lambda$ is $1$-Gorenstein, and is not self-injective.
For an example, please see the algebra $\Lambda$ in Example \ref{example Auslander algebra}.
\end{example}

\begin{example}\label{example 1}
For simple gluing algebra $\Lambda= KQ/I$ of $A$ and $B$, the condition that $I$ is generated by that of $I_A$ and $I_B$ is necessary to ensure the above proposition.

Let $Q_\Lambda$ be the quiver $\xymatrix{1 \ar@/^/[r]^{\alpha} & 2 \ar@/^/[l]^\beta & 3 \ar[l]_\gamma   }$, $Q_A$ be the full subquiver of $Q_\Lambda$ containing vertices $1,2$, and $Q_B$ be the full subquiver of $Q_\Lambda$ containing the vertices $2,3$. Then $Q_\Lambda$ is a simple gluing quiver of $Q_A$ and $Q_B$. Let $I_A=\langle \alpha\beta,\beta\alpha \rangle$, $I_B=0$. Then both $A$ and $B$ are Gorenstein. However, if we take $\Lambda=KQ_\Lambda/I_\Lambda$, where $I_\Lambda$ is generated by $\alpha\beta,\beta\alpha$ and $\beta\gamma$, then $\Lambda$ is not Gorenstein any more.
\end{example}

We recall two basic facts concerning functors between triangulated categories, which are useful to prove our main results.

\begin{lemma}[see. e.g. \cite{chen2}]\label{lemma functor to derived functor}
Let $F_1:\ca\rightarrow\cb$ be an exact functor between abelian categories which has an exact right adjoint $F_2$. Then the pair $(D^b(F_1),D^b(F_2))$ is adjoint, where $D^b(F_1)$ is the induced functor from $D^b(\ca)$ to $D^b(\cb)$ ($D^b(F_2)$ is defined similarly).
\end{lemma}

Since all the six functors defined above are exact, so they induces six triangulated functors on derived categories, and $(D^b(i_\lambda), D^b(i_\mu),D^b(i_\rho))$ and
 $(D^b(j_\lambda), D^b(j_\mu),D^b(j_\rho))$ are adjoint triples. Since $i_\lambda$, $i_\mu$, $j_\lambda$ and $j_\mu$ preserve projectives, they induce triangulated functors on singularity categories, which are denoted by $\tilde{i}_\lambda$, $\tilde{i}_\mu$, $\tilde{j}_\lambda$ and $\tilde{j}_\mu$ respectively.

The second one on adjoint functors is well-known.
\begin{lemma}[\cite{Or1, Chen1}]\label{lemma adjoint}
Let $\cm$ and $\cn$ be full triangulated subcategories in triangulated categories $\cc$ and $\cd$ respectively. Let $F: \cc\to\cd$ and $G: \cd\to\cc$ be an adjoint pair of exact functors such that $F(\cm)\subset \cn$ and $G(\cn)\subset\cm$. Then they induce functors
\[\ow{F}: \cc/\cm\to \cd/\cn, \indent \ow{G}: \cd/\cn\to\cc/\cm
\]which are adjoint.
\end{lemma}

\begin{lemma}\label{lemma adjoint functor singularity category}
Keep the notations as above. Then the following hold:

(i) $(\tilde{i}_\lambda,\tilde{i}_\mu)$, $(\tilde{j}_\lambda,\tilde{j}_\mu)$ are adjoint pairs;

(ii) $\tilde{i}_\lambda$ and $\tilde{j}_\lambda$ are fully faithful;

(iii) $\ker(\tilde{j}_\mu)\simeq \Im (\tilde{i}_\lambda)$ and $\ker (\tilde{i}_\mu)\simeq \Im(\tilde{j}_\lambda)$.
\end{lemma}
\begin{proof}
(i) follows from Lemma \ref{lemma adjoint} directly.

(ii) we only need to prove that $\tilde{i}_\lambda$ is fully faithful.

For any $A$-module $M$, we have that $i_\mu i_\lambda(M)=M \oplus P(a)^{\oplus s}$ for some integer $s$. So $\tilde{i}_\mu \tilde{i}_\lambda(M)\cong M$, in particular, this isomorphism is natural which implies that $\tilde{i}_\mu \tilde{i}_\lambda \simeq \Id_{D_{sg}(A)}$ since $\mod A$ is a generator of $D_{sg}(A)$. Therefore, $\tilde{i}_\lambda$ is fully faithful.

(iii) for any $A$-module $M$, we have that $j_\mu i_\lambda (M)= Q(b)^{\oplus t}$ for some integer $t$. So $\tilde{j}_\mu \tilde{i}_\lambda(M)\cong 0$, and then $\tilde{j}_\mu \tilde{i}_\lambda=0$.
It follows that $\Im (\tilde{i}_\lambda)\subseteq\ker(\tilde{j}_\mu)$. Similarly, we can get that $\Im (\tilde{j}_\lambda)\subseteq\ker(\tilde{i}_\mu)$.

For any $M\in \ker(\tilde{j}_\mu)$, the adjunction $\tilde{i}_\lambda \tilde{i}_\mu (M)\xrightarrow{\epsilon_M} M$ is extended to a triangle in $D_{sg}(\Lambda)$:
$$X\rightarrow\tilde{i}_\lambda \tilde{i}_\mu (M)\xrightarrow{\epsilon_M} M\rightarrow X[1]$$
where $[1]$ is the suspension functor. Applying $\tilde{i}_\mu$ to the above triangle, we get that $X\in \ker (\tilde{i}_\mu)$ since $\tilde{i}_\mu(\epsilon_M)$ is an isomorphism.
Also applying $\tilde{j}_\mu$ to the above triangle, we get that $X\in \ker (\tilde{j}_\mu)$ since $M\in \ker (\tilde{j}_\mu)$ and $\tilde{j}_\mu \tilde{i}_\lambda=0$.

From Lemma \ref{lemma existence of short exact sequene 1}, it is easy to see that $X\cong \tilde{i}_\lambda \tilde{i}_\mu(X) \oplus \tilde{j}_\lambda \tilde{j}_\mu (X)$ in $D_{sg}(\Lambda)$. Together with $X \in \ker (\tilde{i}_\mu) \bigcap\ker (\tilde{j}_\mu)$, we get that $X\cong 0$ in $D_{sg}(\Lambda)$ and then $M\cong \tilde{i}_\lambda \tilde{i}_\mu (M)$, so $M\in \Im(\tilde{i}_\lambda)$. Therefore, $\ker(\tilde{j}_\mu)\simeq \Im (\tilde{i}_\lambda)$ since $\tilde{i}_\lambda$ is fully faithful. Similarly, we can get that $\ker (\tilde{i}_\mu)\simeq \Im(\tilde{j}_\lambda)$.
\end{proof}
It is worth noting that $i_\lambda$ and $j_\lambda$ are faithful functors (in general, not full), which can be obtained from the proof of Lemma \ref{lemma adjoint functor singularity category}.

\begin{theorem}\label{theorem singularity categories}
Let $\Lambda$ be a simple gluing algebra of the two finite-dimensional bound quiver algebras $A$ and $B$. Then $D_{sg}(\Lambda)\simeq D_{sg}(A) \coprod D_{sg}(B)$.
\end{theorem}
\begin{proof}
For any $L\in D_{sg}(A),N\in D_{sg}(B)$, $\Hom_{D_{sg}(\Lambda)}( \tilde{i}_\lambda (L), \tilde{j}_\lambda(N))\cong \Hom_{D_{sg}(A)}(L, \tilde{i}_\mu \tilde{j}_\lambda(N))=0$ since $\tilde{i}_\mu \tilde{j}_\lambda=0$.
Similarly, $\Hom_{D_{sg}(\Lambda)}( \tilde{j}_\lambda(N),\tilde{i}_\lambda (L))=0$.

For any $M\in D_{sg}(\Lambda)$, from Lemma \ref{lemma existence of short exact sequene 1}, we get that $M\cong \tilde{i}_\lambda \tilde{i}_\mu(M) \oplus \tilde{j}_\lambda \tilde{j}_\mu (M)$ in $D_{sg}(\Lambda)$.
Together with the fully faithfulness of $\tilde{i}_\lambda$ and $\tilde{j}_\lambda$, we get that $D_{sg}(\Lambda)\simeq  \Im(\tilde{i}_\lambda) \coprod \Im(\tilde{j}_\lambda)\simeq D_{sg}(A) \coprod D_{sg}(B)$.
\end{proof}

\begin{example}\label{example 2}
For simple gluing algebra $\Lambda= KQ_\Lambda/I_\Lambda$ of $A$ and $B$, the condition that $I_\Lambda$ is generated by $I_A$ and $I_B$ is necessary to ensure the above theorem.

Let $Q_A,Q_B$ and $Q_\Lambda$ be the quivers as the following figure shows.
Then $Q_\Lambda$ is a simple gluing quiver of $Q_A$ and $Q_B$.
Let $A=KQ_A/ \langle  \varepsilon_1^2\rangle$, $B=KQ_B/\langle \varepsilon_2^2 \rangle$.
If $\Lambda=KQ_\Lambda/I_\Lambda$, where $I_\Lambda$ is generated by
$\varepsilon_1^2,\varepsilon_2^2, \varepsilon_1\alpha-\alpha \varepsilon_2$, then all these algebras are Gorenstein algebras.
It is easy to see that $D_{sg}(A)\simeq D_{sg}(B)$, and $D_{sg}(\Lambda)\simeq D^b(KQ)/ [1]$, where $Q$ is a quiver of type $\A_2$, and $[1]$ is the suspension functor, see \cite{RZ}. However, $D_{sg}(\Lambda)$ is not equivalent to $D_{sg}(A)\coprod D_{sg}(B)$.
\begin{center}\setlength{\unitlength}{0.7mm}
 \begin{picture}(50,10)(0,10)
 \put(5,0){\begin{picture}(50,10)
\put(0,-2){$b$}

\qbezier(-1,1)(-3,3)(-2,5.5)
\qbezier(-2,5.5)(1,9)(4,5.5)
\qbezier(4,5.5)(5,3)(3,1)
\put(3.1,1.4){\vector(-1,-1){0.3}}

\put(1,10){$\varepsilon_2$}

\end{picture}}

 \put(40,0){\begin{picture}(50,10)
\put(0,-2){$1$}
\put(18,0){\vector(-1,0){14}}
\put(10,-4){$\alpha$}
\put(20,-2){$2$}

\qbezier(-1,1)(-3,3)(-2,5.5)
\qbezier(-2,5.5)(1,9)(4,5.5)
\qbezier(4,5.5)(5,3)(3,1)
\put(3.1,1.4){\vector(-1,-1){0.3}}

\qbezier(19,1)(17,3)(18,5.5)
\qbezier(18,5.5)(21,9)(24,5.5)
\qbezier(24,5.5)(25,3)(23,1)
\put(23.1,1.4){\vector(-1,-1){0.3}}

\put(1,10){$\varepsilon_1$}
\put(21,10){$\varepsilon_2$}

\end{picture}}

 \put(-40,0){\begin{picture}(50,10)
\put(0,-2){$1$}
\put(18,0){\vector(-1,0){14}}
\put(10,-4){$\alpha$}
\put(20,-2){$a$}

\qbezier(-1,1)(-3,3)(-2,5.5)
\qbezier(-2,5.5)(1,9)(4,5.5)
\qbezier(4,5.5)(5,3)(3,1)
\put(3.1,1.4){\vector(-1,-1){0.3}}

\put(1,10){$\varepsilon_1$}

\end{picture}}

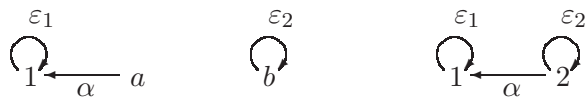
\put(-50,-13){Figure 3. The quivers $Q_A$, $Q_B$ and $Q_\Lambda$ in Example \ref{example 2}. }
\end{picture}

\vspace{1.9cm}
\end{center}
\end{example}

\begin{lemma}[\cite{Lu}]\label{lemma adjoint preserves Gorenstein projectives}
Let $A_1$ and $A_2$ be Artin algebras. If $F_1:\mod A_1\rightarrow \mod A_2$ has the property that it is exact, preserves projective objects, and admits a right adjoint functor $F_2$, then

(i) for any $X\in\mod A_1$ and $Y\in \mod A_2$, we have that $\Ext_{A_2}^k(F_1(X),Y) \cong \Ext_{A_1}^k (X,F_2(Y))$ for any $k\geq1$.

(ii) if $\pd F_2(Q)<\infty$ or $\ind F_2(Q)<\infty$ for any indecomposable projective $A_2$-module $Q$, then $F_1(\Gp(A_1)) \subseteq \Gp(A_2)$.

\end{lemma}

\begin{lemma}\label{lemma the upper four functors presereve Gorenstein projectives}
Keep the notations as above. Then $i_\lambda$, $i_\mu$, $j_\lambda$ and $j_\mu$ preserve Gorenstein projective modules.
\end{lemma}
\begin{proof}
We only prove that for $i_\lambda$ and $i_\mu$.

Since $(i_\lambda,i_\mu)$ is an adjoint pair, and both of them are exact and preserve projectives, Lemma \ref{lemma adjoint preserves Gorenstein projectives} implies that $i_\lambda$ preserves Gorenstein projective modules.

For $i_\mu$, first, we claim that  $i_\mu M\in \,^\bot A$ for any Gorenstein projective $\Lambda$-module $M$. In fact, for any indecomposable Gorenstein projective $\Lambda$-module $M$, Lemma \ref{lemma existence of short exact sequene 1} shows that there is a short exact sequence
$$0\rightarrow U(v)^{\oplus\dim M_v} \rightarrow i_\lambda i_\mu M\oplus j_\lambda j_\mu M\rightarrow M\rightarrow0,$$
which is split since $M\in\,^\bot \Lambda$, and then $i_\lambda i_\mu M\oplus j_\lambda j_\mu M\cong M\oplus U(v)^{\oplus\dim M_v}$. Note that
both $i_\lambda i_\mu M$ and $j_\lambda j_\mu M$ are Gorenstein projective.
Since $M$ is indecomposable, we get that $M\in \add (i_\lambda i_\mu M)$ or $M\in \add(j_\lambda j_\mu M)$. If $M\in \add(j_\lambda j_\mu M)$, then
$i_\lambda i_\mu M$ is projective, which implies that $i_\mu i_\lambda i_\mu M\in \proj A$. On the other hand, $i_\mu i_\lambda i_\mu M =i_\mu  M \oplus P(a)^{\oplus s}$ for some integer $s$, so $i_\mu M$ is also projective.

If $M\in\add (i_\lambda i_\mu M)$, then for any indecomposable projective $\Lambda$-module $U(i)$ with $i\in\ver(Q_A)$, Lemma \ref{lemma adjoint preserves Gorenstein projectives} shows that
$$\Ext_A^k(i_\mu M, i_\mu U(i))\cong \Ext_\Lambda^k(i_\lambda i_\mu M,U(i))=0$$ for any $k>0$,
since $i_\lambda i_\mu M$ is Gorenstein projective. Furthermore, $i_\mu U(i)= P(i)\oplus P(a)^{\oplus t}$ for some integer $t$, which implies that
$\Ext_\Lambda^k(i_\mu M, P(i))=0$ for any $k>0$ and any indecomposable projective $A$-module $P(i)$, and then
$i_\mu M\in \,^\bot A$.

Second, for any Gorenstein projective $\Lambda$-module $Z$, there is an exact sequence
$$0\rightarrow Z \rightarrow U_0 \xrightarrow{d_0} U_1  \xrightarrow{d_1}\cdots$$
with $U_j$ projective and $\ker(d_j) \in \Gp(\Lambda)$ for any $j\geq0$. By applying $i_\mu$ to it, we get that there is an exact sequence
$$0\rightarrow i_\mu Z\rightarrow i_\mu U_0 \xrightarrow{i_\mu d_0} i_\mu U_1\xrightarrow{i_\mu d_1}\cdots $$
with $i_\mu U_j$ projective for any $j\geq0$ since $i_\mu$ is exact and preserves projectives. Additionally, since $\ker(d_i) \in \Gp(\Lambda)$, from the above, we know that $i_\mu \ker(d_j) \in \,^\bot A$ for any $j\geq0$, and then $i_\mu Z$ is Gorenstein projective.
\end{proof}

For any additive category $\ca$, we denote by $\Ind \ca$ the set formed by all the indecomposable objects (up to isomorphisms) of $\ca$.
\begin{theorem}\label{theorem stable category of Cm modules}
Let $\Lambda$ be a simple gluing algebra of the two finite-dimensional bound quiver algebras $A$ and $B$.
Then $\underline{\Gp}(\Lambda)\simeq \underline{\Gp}(A) \coprod \underline{\Gp}(B)$.
Furthermore, for any indecomposable $\Lambda$-module $Z$, $Z$ is Gorenstein projective if and only if there exists an indecomposable Gorenstein projective $A$-module $X$ or $B$-module $Y$ such that $Z\cong i_\lambda(X)$ or $Z\cong j_\lambda(Y)$. \end{theorem}
\begin{proof}
Lemma \ref{lemma the upper four functors presereve Gorenstein projectives} shows that $i_\lambda,i_\mu,j_\lambda,j_\mu$ induce four exact functors on the stable categories of Gorenstein projective modules, which are also denoted by $\tilde{i}_\lambda,\tilde{i}_\mu, \tilde{j}_\lambda,\tilde{j}_\mu$ respectively. Note that $\tilde{i}_\lambda$ and $\tilde{j}_\lambda$ are fully faithful.

For any Gorenstein projective $\Lambda$-modules $Z$, Lemma \ref{lemma existence of short exact sequene 1} yields that
$Z\cong \tilde{i}_\lambda \tilde{i}_\mu(Z) \oplus \tilde{j}_\lambda \tilde{j}_\mu (Z)$.
Similar to the proof of Theorem \ref{theorem singularity categories}, we get that
for any $L\in \Gp(A),N\in \Gp(B)$, $\Hom_{\underline{\Gp}(\Lambda)}( \tilde{i}_\lambda (L), \tilde{j}_\lambda(N))=0=\Hom_{\underline{\Gp}(\Lambda)}( \tilde{j}_\lambda(N),\tilde{i}_\lambda (L))$. Therefore,
$$\underline{\Gp}(\Lambda)\simeq \underline{\Gp}(A) \coprod \underline{\Gp}(B).$$

Obviously, $i_\lambda,j_\lambda$ preserve indecomposable modules. So for any indecomposable Gorenstein projective $A$-module $X$ and $B$-module $Y$, we get that
$i_\lambda(X),j_\lambda(Y)$ are indecomposable Gorenstein projective $\Lambda$-modules. On the other hand, for any indecomposable Gorenstein projective $\Lambda$-module $Z$, if $Z$ is projective, then $Z=U(i)$ for some vertex $i$. If $i\in \ver(Q_A)$, then $Z=i_\lambda(P(i))$; if $i\in \ver(Q_B)$, then $Z=j_\lambda(Q(i))$.

If $Z$ is not projective, then the proof of Lemma \ref{lemma the upper four functors presereve Gorenstein projectives} implies that one and only one of $i_\mu Z$, $j_\mu Z$ is not projective. If $i_\mu Z$ is not projective, then $Z\cong i_\lambda i_\mu Z$ in $\underline{\Gp}(\Lambda)$. Set $i_\mu Z= \bigoplus_{i=1}^nZ_i$, where $Z_i$ is indecomposable for each $1\leq i\leq n$. Then $i_\lambda i_\mu Z=\bigoplus_{i=1}^n i_\lambda(Z_i)$ which is isomorphic to $Z$ in $\underline{\Gp}(\Lambda)$. So there is one and only one non-projective $Z_{i_0}$ such that $i_\lambda(Z_{i_0})\cong Z$. Since $i_\mu(Z)\in \Gp(A)$, we get that $Z_{i_0}$ is an indecomposable Gorenstein projective $A$-module such that $Z\cong i_\lambda(Z_{i_0})$.

If $j_\mu Z$ is not projective, we can prove it similarly.
\end{proof}

\begin{corollary}
Let $\Lambda$ be a simple gluing algebra of several finite-dimensional bound quiver algebras $A_1,\dots,A_n$. If $A_1,\dots A_n$ are self-injective algebras, then
$$D_{sg}(\Lambda)\simeq \underline{\Gproj} \Lambda\simeq \coprod_{i=1}^n\underline{\mod} A_i.$$
\end{corollary}
\begin{proof}
It follows from Proposition \ref{lemma simple gluing Nakayama algebra Gorenstein} and Theorem \ref{theorem stable category of Cm modules} immediately.
\end{proof}

\begin{corollary}\label{corollary Gorenstein defect categories}
Let $\Lambda$ be a simple gluing algebra of the two finite-dimensional bound quiver algebras $A$ and $B$.
Then $D_{def}(\Lambda)\simeq D_{def}(A) \coprod D_{def}(B)$.
\end{corollary}
\begin{proof}
From Theorem \ref{theorem singularity categories} and Theorem \ref{theorem stable category of Cm modules}, we get that
$$D_{sg}(\Lambda)\simeq D_{sg}(A)\coprod D_{sg}(B),\mbox{ and } \underline{\Gp}(\Lambda)\simeq \underline{\Gp}(A) \coprod \underline{\Gp}(B).$$
In particular, since these two equivalences are induced by the functors $i_\lambda,j_\lambda$, they are compatible, which implies the result by the definition of Gorenstein defect categories immediately.
\end{proof}

\begin{corollary}\label{corollary factors through the gluing vertex}
Keep the notations as above. Then for any $X\in\Gp(A)$, $Y\in\Gp (B)$, and any morphism $f\in\Hom_\Lambda(i_\lambda (X),j_\lambda (Y))$, we have that $f$ factors through $U(v)^{\oplus n}$ for some $n$.
\end{corollary}
\begin{proof}
From Theorem \ref{theorem stable category of Cm modules}, we get that $f$ factors through some projective $\Lambda$-module $U$ as $f=f_2f_1$, where $f_1:i_\lambda(X)\rightarrow U$, $f_2:U\rightarrow j_\lambda(Y)$.
Let $Q_Y\xrightarrow{\alpha_1} Y$ be the projective cover of $Y$. Then $f_2$ factors through $j_\lambda(Q_Y)$ as $f_2=j_\lambda(\alpha_1) f_3$ for some $f_3:U\rightarrow j_\lambda(Q_Y)$.
Since $X$ is Gorenstein projective, there is a short exact sequence
$$0\rightarrow X\xrightarrow{\beta_1} P_X\xrightarrow{\beta_2} X_1\rightarrow0$$
such that $P_X$ projective, and $X_1\in\Gp(A)$. Then
$f_1$ factors through $i_\lambda(P_X)$ as $f_1=f_4 i_\lambda(\beta_1)$ for some $f_4: i_\lambda(P_X)\rightarrow U$
since $i_\lambda$ is exact and preserves Gorenstein projective modules.
So $f=j_\lambda(\alpha_1)f_3f_4 i_\lambda(\beta_1)$.
For any indecomposable $\Lambda$-modules $U(i),U(j)$ with $i\in \ver(Q_A)$, $j\in \ver(Q_B)$, and any morphism $\gamma:U(i)\rightarrow U(j)$, it is easy to see that $\gamma$ factors through some object in $\add U(v)$.
So $f_3f_4: i_\lambda(P_X)\rightarrow j_\lambda(Q_Y)$ factors through $U(v)^{\oplus n}$ for some $n$, and then the desired result follows.
\end{proof}

Similarly, for any $X\in\Gp(A)$, $Y\in\Gp (B)$, and any morphism $g\in\Hom_\Lambda(j_\lambda (Y),i_\lambda (X))$, then $g$ factors through $U(v)^{\oplus n}$ for some $n$.

\begin{example}
For simple gluing algebra $\Lambda= KQ/I$ of $A$ and $B$, the condition that $I$ is generated by $I_A$ and $I_B$ is necessary to ensure that the above corollary holds.

Keep the notations as in Example \ref{example 1}. Then $D_{def}(A)=0=D_{def}(B)$ since $A$ and $B$ are Gorenstein. However, $D_{def}(\Lambda)\neq 0$ since $\Lambda$ is not Gorenstein.
\end{example}

The following example is taken from \cite{chen3}.
\begin{example}\label{example not gentle algebra}
Let $\Lambda=KQ/I$ be the algebra where $Q$ is the quiver $\xymatrix{1 \ar@/^/[r]^{\alpha_1} & 2 \ar@/^/[l]^{\beta_1} \ar@/^/[r]^{\alpha_2} & 3\ar@/^/[l]^{\beta_2}  &4,\ar[l]_\gamma   }$
and $I=\langle \beta_i\alpha_i,\alpha_i\beta_i, \beta_2\gamma |i=1,2 \rangle$. Let $A=(e_1+e_2)\Lambda (e_1+e_2)$, and $B=(1-e_1)\Lambda(1-e_1)$. Then $\Lambda$ is a simple gluing algebra of $A$ and $B$. Obviously, $A$ is self-injective with the indecomposable non-projective Gorenstein projective modules $_AS(1),\,_AS(2)$; $B$ is CM-free, and then $D_{def}(B)=D_{sg}(B)$. Theorem \ref{theorem singularity categories} shows that
$D_{sg}(\Lambda)\simeq D_{sg}(A) \coprod D_{sg}(B)$.
From Theorem \ref{theorem stable category of Cm modules}, we get that $\underline{\Gp}(\Lambda)\simeq \underline{\Gp}(A)=\underline{\mod} A$. In particular, the indecomposable non-projective Gorenstein projective $\Lambda$-modules are $i_\lambda(_AS_1)=\,_\Lambda S_1$, $i_\lambda(_A S_2)=\rad U(1)$ (which is the string module with its string $2\xrightarrow{\alpha_2}3$).
From Corollary \ref{corollary Gorenstein defect categories}, we get that $D_{def}(\Lambda)\simeq D_{sg}(B)$.

\end{example}

The following corollary follows from Theorem \ref{theorem stable category of Cm modules} directly.

\begin{corollary}\label{corollary CM-finite CM-free}
Let $\Lambda$ be a simple gluing algebra of the two finite-dimensional bound quiver algebras $A=KQ_A/I_A$ and $B=KQ_B/I_B$. Then

(i) $\Lambda$ is CM-free if and only if $A$ and $B$ are CM-free;

(ii) $\Lambda$ is CM-finite if and only if $A$ and $B$ are CM-finite.
\end{corollary}
At the end of this section, we describe the singularity categories of cluster-tilted algebras of type $\A_n$ and endomorphism algebras of maximal rigid objects of cluster tube $\cc_n$. Note that both of them are $1$-Gorenstein algebras.

Let $H_1=K(\circ\rightarrow\circ)$ be the hereditary algebra of type $\A_2$, $H_2=KQ/I$ the self-injective algebra, where $Q$ is the quiver
\[\xymatrix{&\circ\ar[dr]^{\alpha_2} &\\
\circ\ar[ur]^{\alpha_1} &&\circ\ar[ll]^{\alpha_3}&\mbox{ and }I=\langle \alpha_2\alpha_1,\alpha_3\alpha_2,\alpha_1\alpha_3 \rangle.
}\]

Let $A_1,\dots,A_n$ be algebras where each of them is either $H_1$ or $H_2$. The cluster-tilted algebra of type $\A$ is an algebra $\Lambda$ which is a simple gluing algebra of $A_1,\dots,A_n$ such that there are only two of $A_1,\dots,A_n$ glued at each glued vertex, see \cite{BV}. From Theorem \ref{theorem stable category of Cm modules}, we get the following result immediately.

\begin{corollary}[\cite{CGLu,Ka}]
Let $KQ/I$ be a cluster-tilted algebra of type $\A$. Then
$$\underline{\Gproj}(KQ/I)\simeq \coprod_{t(Q)} \underline{\mod} H_2,$$
where $t(Q)$ is the number of the oriented cycles of length $3$ in $Q$.
\end{corollary}

Let $KQ_1/I_1$ be a cluster-tilted algebra of type $\A_{n-1}$ $(n\geq2)$, and $KQ_2/\langle \varphi^2\rangle$, where $Q_2$ is a loop $\varphi$. Note that $KQ_1/I_1$ is a simple gluing algebra of several algebras which are isomorphic to $H_1$ or $H_2$.
Choose a vertex $a\in Q_1$ which is not a glued vertex for $Q_1$. Then define $KQ/I$ to be the simple gluing algebra of $KQ_1/I_1$ and $KQ_2/I_2$ by identifying $a$ and the unique vertex of $Q_2$. From \cite{Va,Yang}, we get that $KQ/I$ is the endomorphism algebra of a maximal rigid object of cluster tube $\cc_n$, and all endomorphism algebras of maximal rigid objects of cluster tube $\cc_n$ are arisen in this way. Theorem \ref{theorem stable category of Cm modules} yields the following result immediately.

\begin{corollary}[\cite{Ka}]
Let $KQ/I$ be an endomorphism algebra of a maximal rigid objects of cluster tube $\cc_n$. Then
$$\underline{\Gproj}(KQ/I)\simeq \coprod_{t(Q)} \underline{\mod} H_2 \coprod \underline{\mod} K[X]/\langle X^2\rangle.$$
where $t(Q)$ is the number of the oriented cycles of length $3$ in $Q$.
\end{corollary}

\section{Cohen-Macaulay Auslander algebras of simple gluing algebras}
In this section, we describe the Cohen-Macaulay Auslander algebras for CM-finite simple gluing algebras.

First, let us recall some results about almost split sequences and irreducible morphisms.

\begin{lemma}[see e.g. \cite{Liu}]\label{lemma relation of almost split sequences and irreducilbe morphisms}
Let $\ca$ be a Krull-Schmidt exact $K$-category with $0\rightarrow X\xrightarrow{f} Y \xrightarrow{g} Z\rightarrow0$ an almost split sequence.

(i) Up to isomorphism, the sequence is the unique almost split sequence starting with $X$ and the unique one ending with $Z$.

(ii) Each irreducible morphism $f_1:X\rightarrow Y_1$ or $g_1:Y_1\rightarrow Z$ fits into an almost split sequence
$$0\rightarrow X\xrightarrow{\left( \begin{array}{c} f_1\\ f_2 \end{array}\right)} Y_1\oplus Y_2\xrightarrow{(g_1,g_2)}Z\rightarrow0.$$
\end{lemma}

Recall that for a CM-finite algebra $\Lambda$, $\Gproj\Lambda$ is a functorially finite subcategory of $\mod \Lambda$, which implies that $\Gproj\Lambda$ has almost split sequences, see \cite[Theorem 2.4]{AS}, and then $\underline{\Gproj}\Lambda$ has Auslander-Reiten triangles. $\Aus(\Gproj\Lambda)$ is isomorphic to the opposite algebra of $KQ^{\Aus}/I^{\Aus}$ for some ideal $I^{\Aus}$, see Chapter VII, Section 2 of \cite{ARS}.
For any irreducible morphism $f:U(i)\rightarrow U(j)$ in $\Gp \Lambda$, then it is irreducible in $\proj \Lambda$, and there exists some arrow $\alpha:j\rightarrow i$ such that $f$ is induced by $\alpha$.

In the following, we always assume that the simple gluing algebra $\Lambda$ is CM-finite. Since $\Lambda$ is CM-finite, Corollary \ref{corollary CM-finite CM-free} yields that both $A$ and $B$ are CM-finite.
Let $Q^{\Aus}_A$, $Q^{\Aus}_B$ and $Q^{\Aus}_\Lambda$ be the Auslander-Reiten quivers of $\Gproj A$, $\Gp B$ and $\Gp \Lambda$ respectively. First, we prove that $Q^{\Aus}_\Lambda$ is a simple gluing quiver of $Q^{\Aus}_A$ and $Q^{\Aus}_B$ by identifying the vertices corresponding to the indecomposable projective modules $P(a)$ and $Q(b)$. Before that, we give two lemmas.

\begin{lemma}\label{lemma functors preserve for irreducible morphisms between projective modules}
Let $\Lambda=KQ_\Lambda/I_\Lambda$ be a simple gluing algebra of the two finite-dimensional bound quiver algebras $A=KQ_A/I_A$ and $B=KQ_B/I_B$ by identifying $a\in Q_A$ and $b\in Q_B$. Then for any vertices $i,j$ in $Q_\Lambda$ and any arrow $\alpha: j\rightarrow i$ in $Q_\Lambda$, we have

(i) if $\alpha$ is in $Q_A$, and $f: P(i)\rightarrow P(j)$ is the irreducible morphism in $\Gp A$ induced by the arrow $\alpha$, then $i_\lambda(f):U(i)\rightarrow U(j)$ is irreducible.

(ii) if $\alpha$ is in $Q_B$, and $g: P(i)\rightarrow P(j)$ is the irreducible morphism in $\Gp B$ induced by the arrow $\alpha$, then $j_\lambda(g):U(i)\rightarrow U(j)$ is irreducible.

(iii) any irreducible morphism from $U(i)$ to $U(j)$ in $\Gproj\Lambda$ is of form $i_\lambda(f)$
for some irreducible morphism $f: P(i)\rightarrow P(j)$ in $\Gproj A$ or $j_\lambda(g)$ for some irreducible morphism $g: Q(i)\rightarrow Q(j)$ in $\Gp B$.
\end{lemma}
\begin{proof}
(i)
if $i_\lambda(f)$ factors through an object in $\Gp(\Lambda)$, then by Theorem \ref{theorem stable category of Cm modules}, it is of form $i_\lambda(f)=hg$ for some $i_\lambda(P(i))\xrightarrow{g} i_\lambda(L)\oplus j_\lambda(M)\xrightarrow{h} i_\lambda(P(j))$, where $L\in \Gp(A)$, and $M\in\Gp(B)$. We only need to prove that either $g$ is a section or $h$ is a retraction.

Since $f$ is induced by an arrow $\alpha: j\rightarrow i$ in $Q_A$, it is easy to see that $i_\lambda(f): U(i)\rightarrow U(j)$ is the morphism induced by the arrow $\alpha:j\rightarrow i$ in $Q_\Lambda$, which is irreducible in $\proj \Lambda$.
We assume that $L=L_p\oplus L_g$ where $L_p$ is projective, $L_g$ satisfies that all its indecomposable direct summands are non-projective, and
$M=M_p\oplus M_g$ where $M_p$ is projective, $M_g$ satisfies that all its indecomposable direct summands are non-projective.
 So there exist the following exact sequences
$$0\rightarrow L'_g\xrightarrow{k_0} P_0\xrightarrow{k_1} L_g\rightarrow0,\mbox{ and }0\rightarrow L_g\xrightarrow{l_0} P_1\xrightarrow{l_1} L''_g\rightarrow0,$$
with $P_0,P_1\in\proj A$ and $L'_g,L''_g\in\Gproj A$.
Similarly, there are exact sequences
$$0\rightarrow M'_g\xrightarrow{s_0} Q_0\xrightarrow{s_1} M_g\rightarrow0,\mbox{ and }0\rightarrow M_g\xrightarrow{t_0} Q_1\xrightarrow{t_1} M''_g\rightarrow0,$$
with $Q_0,Q_1\in\proj B$ and $M'_g,M''_g\in\Gproj B$.
Since $i_\lambda$ and $j_\lambda$ are exact, preserve projective modules, there exist the following short exact sequences
\begin{eqnarray*}
&&0\rightarrow i_\lambda(L'_g) \oplus j_\lambda(M'_g)\xrightarrow{\tiny\left(\begin{array}{cc}0&0\\i_\lambda(k_0)&0\\
0&0\\ 0& j_\lambda(s_0)\end{array}\right)} i_\lambda(L_p)\oplus i_\lambda(P_0)\oplus j_\lambda(M_p)\oplus j_\lambda(Q_0)\\
&&\xrightarrow{\tiny\left(\begin{array}{cccc}1&&&\\&i_\lambda(k_1)&&\\
&&1&\\ &&& j_\lambda(s_1)\end{array}\right)} i_\lambda(L_p)\oplus i_\lambda(L_g) \oplus j_\lambda(M_p)\oplus j_\lambda(M_g)\rightarrow0,
\end{eqnarray*}
and
\begin{eqnarray*}
&&0\rightarrow  i_\lambda(L_p)\oplus i_\lambda(L_g) \oplus j_\lambda(M_p)\oplus j_\lambda(M_g)\xrightarrow{\tiny\left(\begin{array}{cccc}1&&&\\&i_\lambda(l_0)&&\\
&&1&\\ &&& j_\lambda(t_0)\end{array}\right)} \\
&&i_\lambda(L_p)\oplus i_\lambda(P_1)\oplus j_\lambda(M_p)\oplus j_\lambda(Q_1)\xrightarrow{\tiny\left(\begin{array}{cccc}0&i_\lambda(l_1)&0&0\\0&0&0&j_\lambda(t_1)\end{array}\right)} i_\lambda(L''_g) \oplus j_\lambda(M''_g)\rightarrow0.
\end{eqnarray*}

It is easy to see that there exist morphisms $p: i_\lambda(P(i))\rightarrow i_\lambda(L_p)\oplus i_\lambda(P_0)\oplus j_\lambda(M_p)\oplus j_\lambda(Q_0)$ and
$q: i_\lambda(L_p)\oplus i_\lambda(P_1)\oplus j_\lambda(M_p)\oplus j_\lambda(Q_1)\rightarrow i_\lambda(P(i))$
such that
$$g= \left(\begin{array}{cccc}1&&&\\&i_\lambda(k_1)&&\\
&&1&\\ &&& j_\lambda(s_1)\end{array}\right)p,
\mbox{ and }
h=q\left(\begin{array}{cccc}1&&&\\&i_\lambda(l_0)&&\\
&&1&\\ &&& j_\lambda(t_0)\end{array}\right). $$
Therefore,
$$i_\lambda(f)=hg=q\left(\begin{array}{cccc}1&&&\\&i_\lambda(l_0 k_1) &&\\
&&1&\\ &&& j_\lambda(t_0s_1)\end{array}\right)p.$$
Since $i_\lambda(f)$ is irreducible in $\proj\Lambda$, we get that either $p$ is a section or $q$ is a retraction.

For $p$ is a section, if $p$ induces $i_\lambda(P(i))$ to be a direct summand of $i_\lambda(L_p)$ or a direct summand of $j_\lambda(M_p)$, it is easy to see that $g$ is a section.
If $p$ induces $i_\lambda(P(i))$ to be a direct summand of $i_\lambda(P_g)$, then
$$\left(\begin{array}{cccc}1&&&\\&i_\lambda(l_0k_1)&&\\
&&1&\\ &&& j_\lambda(t_0s_1)\end{array}\right)p$$
is not a section, which yields that $q$ is a retraction, in particular, $q$ induces that $i_\lambda(P(j))$ is a direct summand of $i_\lambda (P_1)$. So $f$ factors through $L_g$ which implies that $f$ is not irreducible in $\Gp A$ since $L_g$ satisfies that all its indecomposable direct summands are non-projective, a contradiction.
If $p$ induces $i_\lambda(P(i))$ to be a direct summand of $j_\lambda(M_g)$, then we get that $i_\lambda(f)$ factors through $j_\lambda(t_0s_1)$. From $t_0s_1:Q_0\rightarrow Q_1$ factors through $M_g$, it is easy to see that there exists an element $w=\sum_{i=0}^n d_i w_i\in KQ_\Lambda$ with $d_i\in K$, $w_i$ a path containing as least one arrow in $Q_B$ for each $1\leq i\leq n$, such that $\alpha-w\in I_\Lambda$, which is impossible.

For $q$ is a retraction, it is dual to the above, we omit the proof here.
So $i_\lambda(f)$ is irreducible.

For (ii), it is similar to (i).

(iii) for any irreducible morphism $\psi$ from $U(i)$ to $U(j)$ in $\Gproj\Lambda$, then it is also irreducible in $\proj\Lambda$, which is induced by an arrow $\alpha:j\rightarrow i$. If $\alpha$ is in $Q_A$, then $\psi=i_\lambda(f)$ where $f$ is the irreducible morphism $f:P(i)\rightarrow P(j)$ in $\proj A$ induced by $\alpha$. Suppose for a contradiction that there exist morphisms $f_1$ and $f_2$ in $\Gproj A$ such that $f=f_2f_1$ with neither $f_1$ a section nor $f_2$ a retraction. Then it is easy to see that neither $i_\lambda (f_1)$ is a section nor $i_\lambda(f_2)$ a retraction, and $i_\lambda(f)=i_\lambda(f_2) i_\lambda(f_1)$, a contradiction to that $i_\lambda(f)$ is irreducible in $\Gp \Lambda$.

If $\alpha$ is in $Q_B$, then we can prove it similarly.
\end{proof}

\begin{lemma}\label{lemma functors preserve for irreducible morphisms between non-projective modules}
Let $\Lambda=KQ_\Lambda/I_\Lambda$ be a simple gluing algebra of the two finite-dimensional bound quiver algebras $A=KQ_A/I_A$ and $B=KQ_B/I_B$ by identifying $a\in Q_A$ and $b\in Q_B$. If $\Lambda$ is CM-finite, then

(i) for any irreducible morphism $f:X\rightarrow Y$ in $\Gp A$ with either $X$ or $Y$ not projective, we have that $i_\lambda(f)$ is irreducible in $\Gproj \Lambda$;

(ii) for any irreducible morphism $g:L\rightarrow M$ in $\Gp B$ with either $L$ or $M$ not projective, we have that $j_\lambda(g)$ is irreducible in $\Gproj \Lambda$;

(iii) any irreducible morphism in $\Gproj\Lambda$ is of form $i_\lambda(f)$
for some irreducible morphism $f: X\rightarrow Y$ in $\Gproj A$ or $j_\lambda(g)$ for some irreducible morphism $g: L\rightarrow M$ in $\Gp B$.
\end{lemma}
\begin{proof}
Since $\Lambda$ is CM-finite, Corollary \ref{corollary CM-finite CM-free} yields that both $A$ and $B$ are CM-finite. Then all the categories $\Gp A$, $\Gp B$ and $\Gp \Lambda$ have almost split sequences.

First, we prove that $i_\lambda$ preserves almost split sequences.
For any almost split sequence
$$0\rightarrow L\xrightarrow{f} M\xrightarrow{g} N\rightarrow0$$
in $\Gp(\Lambda)$, obviously, $L$ and $N$ are indecomposable, and
$$L\xrightarrow{\tilde{f}} M \xrightarrow{\tilde{g}}N\rightarrow \Sigma L$$
is an Auslander-Reiten triangle in $\underline{\Gp}(\Lambda)$, where $\Sigma$ is its suspension functor.
Obviously.
$i_\lambda(L)$ and $i_\lambda(N)$ are indecomposable modules. Since $i_\lambda$ is exact, the sequence
\begin{equation}\label{equation 6}
0\rightarrow i_\lambda (L)\xrightarrow{i_\lambda(f)} i_\lambda(M)\xrightarrow{i_\lambda(g)} i_\lambda(N)\rightarrow0
\end{equation}
is exact. We claim that the sequence (\ref{equation 6}) is almost split. From Theorem \ref{theorem stable category of Cm modules}, we get that
$\tilde{i}_\lambda$ preserves Auslander-Reiten triangles, and then
\begin{equation}\label{equation 7}
i_\lambda (L)\xrightarrow{\widetilde{i_\lambda(f)}} i_\lambda(M) \xrightarrow{\widetilde{i_\lambda(g)}}i_\lambda(N)\rightarrow \Sigma L
\end{equation}
is an Auslander-Reiten triangles in $\underline{\Gp}(\Lambda)$.
For any morphism $h:Z\rightarrow i_\lambda(N)$ which is not a retraction, then in $\underline{\Gp}(\Lambda)$, $\tilde{h}$ is also not a retraction. The Auslander-Reiten triangle (\ref{equation 7}) yields that $\tilde{h}$ factors through $\widetilde{i_\lambda(g)}$ as $\tilde{h}=\widetilde{i_\lambda(g)}\tilde{l}$ for some morphism $l: Z\rightarrow i_\lambda(M)$.
Then there exists a projective $\Lambda$-module $U$ such that $h=i_\lambda(g) l+p_2p_1$ for some morphism $p_1:Z\rightarrow U$ and $p_2:U\rightarrow i_\lambda(N)$ as the following diagram shows
\[\xymatrix{ i_\lambda(L) \ar[r]^{i_\lambda(f)} & i_\lambda(M) \ar[r]^{i_\lambda(g)}& i_\lambda(N) \\
&Z\ar[r]^{p_1} \ar[u]^l \ar[ur]^{h} &U.\ar[u]^{p_2} }\]
Since $U$ is projective, we get that $p_2$ factors through $i_\lambda(g)$ as $p_2i_\lambda(g) p_3$ for some morphism $p_3:U\rightarrow i_\lambda(M)$. Then
$h=i_\lambda(g) l+p_2p_1=i_\lambda(g)( l+p_3p_1)$ which implies that $i_\lambda(g)$ is a right almost split morphism. Therefore, the sequence (\ref{equation 6})
is almost split in $\Gproj\Lambda$.

Similarly, we get that $j_\lambda$ preserves almost split sequences.

Theorem \ref{theorem stable category of Cm modules} also implies that for any Auslander-Reiten triangles in $\underline{\Gproj}\Lambda$, this triangle is either from an Auslander-Reiten triangle in $\underline{\Gp}(A)$ mapped by $\tilde{i}_\lambda$ or from an Auslander-Reiten triangle in $\underline{\Gp}(B)$ mapped by $\tilde{j}_\lambda$.
Note that all the Auslander-Reiten triangles in the stable categories are induced by almost split sequences. So any almost split sequence in $\Gp(\Lambda)$ is either of form
$$0\rightarrow i_\lambda(L)\xrightarrow{i_\lambda(u)} i_\lambda(M) \xrightarrow{i_\lambda(v)} i_\lambda(N)\rightarrow0 $$
where $0\rightarrow L\xrightarrow{u} M\xrightarrow{v}N\rightarrow0$ is an almost split sequence in $\Gp(A)$;
or of form $$0\rightarrow j_\lambda(X)\xrightarrow{j_\lambda(u)} j_\lambda(Y) \xrightarrow{j_\lambda(v)} j_\lambda(Z)\rightarrow0 $$
where $0\rightarrow X\xrightarrow{u} Y\xrightarrow{v}Z\rightarrow0$ is an almost split sequence in $\Gp(B)$.

(i) for any irreducible morphism $f:X\rightarrow Y$ in $\Gproj A$, if $X$ is not projective, then there is an almost split sequence starting with $X$, which is of form $$0\rightarrow X\xrightarrow{\left( \begin{array}{c} f\\ f_2 \end{array}\right)} Y_1\oplus Y_2\xrightarrow{(g_1,g_2)}Z\rightarrow0.$$
by Lemma \ref{lemma relation of almost split sequences and irreducilbe morphisms} (ii). From above, we get an almost split sequence in $\Gproj \Lambda$
$$0\rightarrow i_\lambda X\xrightarrow{\left( \begin{array}{c} i_\lambda(f)\\ i_\lambda(f_2) \end{array}\right)} i_\lambda(Y_1)\oplus i_\lambda(Y_2)\xrightarrow{(i_\lambda(g_1),i_\lambda(g_2))}i_\lambda Z\rightarrow0.$$
From it, it is easy to see that $i_\lambda(f)$ is irreducible. If $Y$ is not projective, we can prove it similarly.

(ii) is similar to (i).

(iii) for any irreducible morphism $\psi: W\rightarrow Z$ in $\Gp \Lambda$, if $W,Z\in \proj \Lambda$, then it follows from Lemma \ref{lemma functors preserve for irreducible morphisms between projective modules} (iii) immediately. Otherwise, if $W$ is not projective, then Theorem \ref{theorem stable category of Cm modules} shows that $W$ is of form $i_\lambda (X)$ or $j_\lambda(L)$ for some non-projective indecomposable Gorenstein projective modules $X\in \Gproj A$ and $L\in \Gproj B$. We only prove for the case $W=i_\lambda(X)$, since the other one is similar.
Then there is an almost split sequence starting with $i_\lambda (X)$, by the above, it is of form
$$0\rightarrow i_\lambda(X)\xrightarrow{i_\lambda(f_1)} i_\lambda(Y_1) \xrightarrow{i_\lambda(g_1)} i_\lambda(Z_1)\rightarrow0.$$
From Lemma \ref{lemma relation of almost split sequences and irreducilbe morphisms} (ii), it is easy to see that $\psi=i_\lambda (f)$ for some $i_\lambda (f):i_\lambda (X)\rightarrow i_\lambda(Y)$, where $Y$ is an indecomposable direct summand of $Y_1$. Since $i_\lambda(f)$ is irreducible, similar to the proof of Lemma \ref{lemma functors preserve for irreducible morphisms between projective modules} (iii), we get that $f$ is irreducible in $\Gp A$.

For the case when $Z$ is not projective, we can prove it similarly.
\end{proof}

Recall that for any CM-finite algebra $\Lambda$, the Auslander-Reiten quiver of $\Gproj \Lambda$ is formed by indecomposable objects and irreducible morphisms in $\Gproj\Lambda$. So we get the following result by Lemma \ref{lemma functors preserve for irreducible morphisms between projective modules} and Lemma \ref{lemma functors preserve for irreducible morphisms between non-projective modules} immediately.
\begin{proposition}\label{proposition simple gluing of AR-quivers}
Let $\Lambda=KQ_\Lambda/I_\Lambda$ be a simple gluing algebra of the two finite-dimensional bound quiver algebras $A=KQ_A/I_A$ and $B=KQ_B/I_B$ by identifying $a\in Q_A$ and $b\in Q_B$.
Then $\Lambda$ is CM-finite if and only if $A$ and $B$ are CM-finite. In this case, the Auslander-Reiten quiver $Q^{\Aus}_\Lambda$ of $\Lambda$ is a simple gluing quiver of the Auslander-Reiten quiver $Q^{\Aus}_A$ of $A$ and the Auslander-Reiten quiver $Q^{\Aus}_B$ of $B$ by identifying the vertices corresponding to the indecomposable projective modules $P(a)$ and $Q(b)$.
\end{proposition}

Recall that for any $M\in\mod A$ and $N\in\mod \Lambda$, we denote by
$$\alpha_{M,N}: \Hom_\Lambda(i_\lambda(M),N)\xrightarrow{\sim} \Hom_A(M,i_\mu(N))$$
the adjoint isomorphism.
For any $M\in \mod A$, we denote by $\mu_M= \alpha_{M, i_\lambda (M)} (1_{i_\lambda (M)})$ the adjunction morphism.
For any $L\in\mod B$ and $N\in\mod \Lambda$ we denote by
$$\beta_{L,N}: \Hom_\Lambda(j_\lambda(L),N)\xrightarrow{\sim} \Hom_B(L,i_\mu(N))$$
the adjoint isomorphism.

\begin{lemma}\label{lemma characterization of ideal 1}
Let $\Lambda=KQ_\Lambda/I_\Lambda$ be a simple gluing algebra of the two finite-dimensional bound quiver algebras $A=KQ_A/I_A$ and $B=KQ_B/I_B$ by identifying $a\in Q_A$ and $b\in Q_B$.
Let $X\in \Gp A$ and $Y\in \Gp B$ be indecomposable Gorenstein projective modules. For any morphisms $f_i:X\rightarrow P(a)$ and $g_i:Q(b)\rightarrow Y$ where $1\leq i\leq n$, if $f_1,\dots,f_n$ are linearly independent and $\sum_{i=1}^n j_\lambda(g_i)i_\lambda(f_i)=0$, then $g_i=0$ for any $1\leq i\leq n$.
\end{lemma}
\begin{proof}
Recall that $i_\mu j_\lambda(Y)= P(a)^{\oplus m}$ for some integer $m$. If $m=0$, then $i_\mu j_\lambda(Y)=0$, so $\alpha_{P(a), j_\lambda(Y)}(j_\lambda(g_i))=0$ for any $1\leq i\leq n$, which implies that $j_\lambda(g_i)=0$, and then
$g_i=0$ for any $1\leq i\leq n$.

If $m>0$, then $\alpha_{X,j_\lambda(Y)} (j_\lambda(g_i)i_\lambda(f_i))=\alpha_{P(a),j_\lambda(Y)}(j_\lambda (g_i)) f_i$ for any $1\leq i\leq n$.
Since $i_\mu j_\lambda(Y)= P(a)^{\oplus m}$, from (\ref{equation form of morphism 2}) and the section map $\mu_{P(a)}$, we get that $\alpha_{P(a),j_\lambda(Y)}(j_\lambda (g_i)) f_i= i_\mu j_\lambda (g_i)\mu_{ P(a)} f_i$ is of form
$$X\xrightarrow{f_i} P(a) \xrightarrow{\left(\begin{array}{c} k_{i1}\\ \vdots\\ k_{im} \end{array} \right)  }P(a)^{\oplus m},$$
where $k_{ij}\in K$ for any $1\leq j\leq m$.

On the other hand, $\alpha_{X, j_\lambda(Y)}(\sum_{i=1}^n j_\lambda(g_i)i_\lambda(f_i))=\sum_{i=1}^n \alpha_{P(a),j_\lambda(Y)}(j_\lambda(g_i)) f_i$, which is equal to
$$ \sum_{i=1}^n\left(\begin{array}{c} k_{i1}\\ \vdots\\ k_{im} \end{array} \right) f_i=0.$$
Since $f_1,\dots,f_n$ are linearly independent, we get that $k_{ij}=0$ for all $1\leq i\leq n$, $1\leq j\leq m$. Therefore, $\alpha_{P(a),j_\lambda(Y)}(g_i)=0$ for any $1\leq i\leq n$, and then $g_i=0$ for any $1\leq i\leq n$.
\end{proof}

Similarly, we get the following lemma.

\begin{lemma}\label{lemma characterization of ideal 2}
Let $\Lambda=KQ_\Lambda/I_\Lambda$ be a simple gluing algebra of the two finite-dimensional bound quiver algebras $A=KQ_A/I_A$ and $B=KQ_B/I_B$ by identifying $a\in Q_A$ and $b\in Q_B$.
Let $X\in \Gp A$ and $Y\in \Gp B$ be indecomposable Gorenstein projective modules. For any morphisms $f_i:P(a)\rightarrow X$ and $g_i:Y\rightarrow Q(b)$ where $1\leq i\leq n$, if $g_1,\dots,g_n$ are linearly independent and $\sum_{i=1}^n i_\lambda(f_i)j_\lambda(g_i)=0$, then $f_i=0$ for any $1\leq i\leq n$.
\end{lemma}

\begin{remark}\label{remark form of combination of morphism}
Let $X\in \Gp A$ and $Y\in \Gp B$ be indecomposable Gorenstein projective modules. Recall that $i_\mu i_\lambda(X)= X\oplus P(a)^{\oplus s}$ and $i_\mu j_\lambda(Y)=P(a)^{\oplus t}$ for some integers $s,t$. For any morphisms $f:P(a)^{\oplus n}\rightarrow X$ and $g: Y \rightarrow Q(b)^{\oplus n}$ with $\Im g\subseteq \rad Q(b)^{\oplus n}$, we have that
$i_\mu(i_\lambda(f)j_\lambda(g)):P(b)^{\oplus t}\rightarrow X\oplus P(a)^{\oplus s}$ is of form $\left(\begin{array}{c}0\\h\end{array}\right)$, where $h$ is a $s\times t$ matrix with entries in $K$.
\end{remark}
\begin{proof}
First, $i_\mu (U(v)^{\oplus n})= P(a)^{\oplus n}\oplus P(a)^{\oplus r}$ for some integer $r$.
From (\ref{equation form of morphism 1}) and (\ref{equation form of morphism 2}), we get that
$i_\mu j_\lambda(g):i_\mu j_\lambda Y=P(a)^{\oplus t}\rightarrow i_\mu (U(v)^{\oplus n})= P(a)^{\oplus n}\oplus P(a)^{\oplus r}$ is of form
$\left(\begin{array}{c}0\\h_2\end{array}\right)$ where $h_2$ is a $r\times t$ matrix with entries in $K$, since $\Im g\subseteq \rad Q(b)^{\oplus n}$, and
$i_\mu i_\lambda (f): i_\mu (U(v)^{\oplus n})= P(a)^{\oplus n}\oplus P(a)^{\oplus r}\rightarrow X\oplus P(a)^{\oplus s}$ is of form
$\left(\begin{array}{cc}f&\\&h_1\end{array}\right)$, where $h_1$ is a $s\times r$ matrix with entries in $K$.
Then $i_\mu(i_\lambda(f)j_\lambda(g))=\left(\begin{array}{c}0\\h_1h_2\end{array}\right)$ where $h_2h_1$ is a $s\times t$ matrix with entries in $K$.

\end{proof}

Now, we get the final main result in this paper.
\begin{theorem}\label{theorem Cohen-Macaulay Auslander algebras}
Let $\Lambda$ be a simple gluing algebra of the two finite-dimensional bound quiver algebras $A=KQ_A/I_A$ and $B=KQ_B/I_B$ by identifying $a\in Q_A$ and $b\in Q_B$. Then $\Lambda$ is CM-finite if and only if $A$ and $B$ are CM-finite. In this case, the Cohen-Macaulay Auslander algebra $\Aus(\Gp(\Lambda))$ is a simple gluing algebra of $\Aus(\Gp(A))$ and $\Aus(\Gp(B))$ by identifying the vertices corresponding to the indecomposable projective modules $P(a)$ and $Q(b)$.
\end{theorem}
\begin{proof}
From Proposition \ref{proposition simple gluing of AR-quivers},
the Auslander-Reiten quiver $Q^{\Aus}_\Lambda$ of $\Lambda$ is a simple gluing quiver of the Auslander-Reiten quiver $Q^{\Aus}_A$ of $A$ and the Auslander-Reiten quiver $Q^{\Aus}_B$ of $B$ by identifying the vertices corresponding to the indecomposable projective modules $P(a)$ and $Q(b)$.

On the other hand, from the above, we also get that all the irreducible morphisms in $\Gp\Lambda$ are induced explicitly by the ones in $\Gp A$ and $\Gp B$. So for any irreducible morphism in $\Gp\Lambda$, it is either of form $i_\lambda(f)$ for some irreducible morphism $f$ in $\Gp A$ or of form $j_\lambda(g)$ for some irreducible morphism $g$ in $\Gp B$.
By viewing $Q^{\Aus}_A$ and $Q^{\Aus}_B$ to be subquivers of $Q^{\Aus}_\Lambda$, we identify $f$ with $i_\lambda(f)$, $g$ with $j_\lambda(g)$ for any irreducible morphisms $f$ in $\Gp A$, and $g$ in $\Gp B$. Let $\Aus(\Gproj A)$, $\Aus(\Gproj B)$ and $\Aus(\Gp \Lambda)$ be the opposite algebras of $KQ^{\Aus}_A/I^{\Aus}_A$, $KQ^{\Aus}_B/I^{\Aus}_B$ and $KQ^{\Aus}_\Lambda/I^{\Aus}_\Lambda$ respectively.
In this way, we claim that $I^{\Aus}_\Lambda= \langle I^{\Aus}_A, I^{\Aus}_B\rangle$.

First, it is easy to see that $I^{\Aus}_A \subseteq I^{\Aus}_\Lambda$, and $I^{\Aus}_B\subseteq I^{\Aus}_\Lambda$, which implies that
$\langle I^{\Aus}_A, I^{\Aus}_B\rangle\subseteq I^{\Aus}_\Lambda$.

Second, for any element $w\in KQ^{\Aus}_\Lambda$ which is in $I^{\Aus}_\Lambda$, without loss of generality, we assume that $w$ starts from an indecomposable Gorenstein projective $\Lambda$-module $Z_0$, and ends to an indecomposable Gorenstein projective $\Lambda$-module $Z_1$.
It is easy to see that $w$ is a linear combination of combinations of irreducible morphisms in $\Gproj \Lambda$ by viewing arrows to be irreducible morphisms. Recall that the Auslander-Reiten quiver $Q^{\Aus}_\Lambda$ is a simple gluing quiver of the Auslander-Reiten quivers $Q^{\Aus}_A$ and $Q^{\Aus}_B$ by identifying the vertices corresponding to the indecomposable projective modules $P(a)$ and $Q(b)$.
The proof can be broken into the following four cases.

Case (a) $Z_0=i_\lambda(X)$, $Z_1=j_\lambda(Y)$ for some $X\in\Ind \Gp A$, $Y\in\Ind\Gp B$, and $w$ is of form
\begin{eqnarray*}
&&i_\lambda(X)\xrightarrow{i_\lambda(f_0)} U(v)^{\oplus n_0} \xrightarrow{ j_\lambda(g_0)} U(v)^{\oplus m_0} \xrightarrow{i_\lambda(f_1)} U(v)^{\oplus n_1}
\xrightarrow{j_\lambda(g_1)} U(v)^{\oplus m_1}\\
&& \xrightarrow{i_\lambda(f_2)}\cdots \xrightarrow{j_\lambda(g_{t-1})} U(v)^{\oplus m_{t-1}} \xrightarrow{ i_\lambda(f_{t})} U(v)^{\oplus n_{t}}  \xrightarrow{j_\lambda(g_t)}j_\lambda(X_t)=j_\lambda(Y),
\end{eqnarray*}
with $\Im (f_i)\subseteq \rad P(a)^{\oplus n_i}$ for any $0\leq i\leq t$ and $\Im (g_j)\subseteq \rad Q(b)^{\oplus m_j}$ for any $0\leq j\leq t-1$.
We prove that $w\in \langle I_A^{\Aus},I_B^{\Aus}\rangle$ by induction on $t$.

If $t=0$, then it follows from Lemma \ref{lemma characterization of ideal 1}.
For $t>0$, let $f$ be of form
$$ X\xrightarrow{ \left(\begin{array}{c} f_{01}\\ \vdots\\ f_{0n_0} \end{array} \right) } P(a)^{\oplus n_0}.$$
If $f_{01}=0,\dots,f_{0n_0}=0$, then $f_0=0$, which implies that $w\in\langle I_A^{\Aus}\rangle \subseteq \langle I_A^{\Aus},I_B^{\Aus}\rangle $. Otherwise, without loss of generality, we assume that $f_{01},f_{02},\dots,f_{0n_0}$ are linearly independent.
Then $w$ is of form
$$i_\lambda(X)\xrightarrow{ \left(\begin{array}{c} i_\lambda(f_{01})\\ \vdots\\ i_\lambda(f_{0n_0}) \end{array} \right) } U(v)^{\oplus n_0}\xrightarrow{(l_1,l_2,\dots,l_{n_0}) } j_\lambda(Y),$$
where $(l_1,l_2,\dots,l_{n_0})=j_\lambda(g_{t})\cdots  i_\lambda(f_1)j_\lambda(g_{0})$.
Furthermore, $$\alpha_{X, j_\lambda(Y)}( \sum_{i=1}^{n_0}  l_ii_\lambda(f_{0i})) =\sum_{i=1}^{n_0}i_\mu (l_i) \mu_{P(a)^{\oplus n_0}}f_{0i}.$$
Since $\Im (g_j)\subseteq \rad Q(b)^{\oplus m_j}$ for any $0\leq j\leq t-1$, from Remark \ref{remark form of combination of morphism} and (\ref{equation form of morphism 2}),
we get that $i_\mu(l_i):P(a)^{\oplus s} \rightarrow P(a)^m$ is represented by a matrix with its entries in $K$, which implies that this property holds for $i_\mu(l_i) \mu_{P(a)^{\oplus n_0}}$ by the form of the section map $\mu_{P(a)^{\oplus n_0}}$.
Similar to the proof of Lemma \ref{lemma characterization of ideal 1}, we get that $i_\mu(l_i) \mu_{P(a)^{\oplus n_0}}=0$ for each $1\leq i\leq n_0$, and then $j_\lambda(g_{t})\cdots  i_\lambda(f_1)j_\lambda(g_{0})=0$.

In order to get that $j_\lambda(g_{t})\cdots  i_\lambda(f_1)j_\lambda(g_{0})\in\langle I_A^{\Aus}, I_B^{\Aus}\rangle$, without loss of generality, we assume $n_0=1$.
Then $g_0$ is of form
$$Q(b)\xrightarrow{\left(\begin{array}{c} g_{01}\\ \vdots\\ g_{0m_0} \end{array} \right)  } Q(b)^{\oplus m_0}.$$
Similar to the above, it is enough to prove that for the case when $g_{01},\dots, g_{0m_0}$ are linear independent.
Let $(p_1,\dots,p_{m_0}): U(v)^{\oplus m_0}\rightarrow j_\lambda(Y)$ be the morphism $j_\lambda(g_{t})\cdots  i_\lambda(f_1)$.
Then
\begin{eqnarray*}
&&\beta_{Q(b),j_\lambda(Y)}(j_\lambda(g_{t})\cdots i_\lambda(f_1)j_\lambda(g_{0}) )\\
&=&j_\mu (j_\lambda(g_{t})\cdots i_\lambda(f_1))\nu_{Q(b)^{\oplus m_0}}g_0=0.
\end{eqnarray*}
Similar to the proof of Remark \ref{remark form of combination of morphism}, since $j_\mu j_\lambda(Y)= Y\oplus Q(b)^{\oplus q}$ for some integer $q$, we get that $j_\mu (j_\lambda(g_{t})\cdots i_\lambda(f_1))\nu_{Q(b)^{\oplus m_0}}: Q(b)^{\oplus m_0}\rightarrow Y\oplus Q(b)^{\oplus q}$ is of form
$\left(\begin{array}{c}0\\ h' \end{array} \right)$ where $h'$ is a $q\times m_0$ matrix with entries in $K$.
Since $g_{01},\dots, g_{0m_0}$ are linear independent,
$$\beta_{Q(b),j_\lambda(Y)}(j_\lambda(g_{t})\cdots i_\lambda(f_1))=j_\mu (j_\lambda(g_{t})\cdots i_\lambda(f_1))\nu_{Q(b)^{\oplus m_0}}=0$$
 and then
$j_\lambda(g_{t})\cdots i_\lambda(f_1)=0$.
From the assumption of induction, we have
$$j_\lambda(g_{t})i_\lambda(f_{t})\cdots j_\lambda(g_{1}) i_\lambda(f_1)\in \langle I^{\Aus}_A,I^{\Aus}_B\rangle,$$
and then
$w\in  \langle I^{\Aus}_A,I^{\Aus}_B\rangle$.

Case (b) $Z_0=j_\lambda(Y)$ and $Z_1=i_\lambda(X)$ for some $X\in\Ind \Gp A$ and $Y\in\Ind\Gp B$.  $Z_0=i_\lambda(X)$, $Z_1=j_\lambda(Y)$ for some $X\in\Ind \Gp A$, $Y\in\Ind\Gp B$, and $w$ is of form
\begin{eqnarray*}
&&j_\lambda(Y)\xrightarrow{j_\lambda(f_0)} U(v)^{\oplus n_0} \xrightarrow{ i_\lambda(g_0)} U(v)^{\oplus m_0} \xrightarrow{j_\lambda(f_1)} U(v)^{\oplus n_1}
\xrightarrow{i_\lambda(g_1)} U(v)^{\oplus m_1}\\
&& \xrightarrow{j_\lambda(f_2)}\cdots \xrightarrow{i_\lambda(g_{t-1})} U(v)^{\oplus m_{t-1}} \xrightarrow{ j_\lambda(f_{t})} U(v)^{\oplus n_{t}}  \xrightarrow{i_\lambda(g_t)}i_\lambda(Y_t)=i_\lambda(X),
\end{eqnarray*}
with $\Im (f_i)\subseteq \rad P(a)^{\oplus n_i}$ for any $0\leq i\leq t$ and $\Im (g_j)\subseteq \rad Q(b)^{\oplus m_j}$ for any $0\leq j\leq t-1$.
It is similar to Case (a).

Case (c) $Z_0=i_\lambda(X)$, $Z_1=i_\lambda(Y)$ for some $X,Y\in\Ind \Gp A$, and $w$ is of form
\begin{eqnarray*}
&&i_\lambda(X)\xrightarrow{i_\lambda(f_0)} U(v)^{\oplus n_0} \xrightarrow{ j_\lambda(g_0)} U(v)^{\oplus m_0} \xrightarrow{i_\lambda(f_1)} U(v)^{\oplus n_1}
\xrightarrow{j_\lambda(g_1)} U(v)^{\oplus m_1}\\
&& \xrightarrow{i_\lambda(f_2)}\cdots \xrightarrow{j_\lambda(g_{t-1})} U(v)^{\oplus m_{t-1}} \xrightarrow{ i_\lambda(f_{t})} i_\lambda(X_t)=i_\lambda(Y),
\end{eqnarray*}
with $\Im (f_i)\subseteq \rad P(a)^{\oplus n_i}$ for any $0\leq i\leq t-1$ and $\Im (g_j)\subseteq \rad Q(b)^{\oplus m_j}$ for any $0\leq j\leq t-1$. If $t=0$, then $i_\lambda(f_0)=0$ and then $f_0=0$, which implies that
$w\in I_A^{\Aus}$. For $t>0$, we can prove it similar to Case (a).

Case (d) $Z_0=j_\lambda(X)$, $Z_1=j_\lambda(Y)$ for some $X,Y\in\Ind \Gp B$, and $w$ is of form
\begin{eqnarray*}
&&j_\lambda(X)\xrightarrow{j_\lambda(f_0)} U(v)^{\oplus n_0} \xrightarrow{ i_\lambda(g_0)} U(v)^{\oplus m_0} \xrightarrow{j_\lambda(f_1)} U(v)^{\oplus n_1}
\xrightarrow{i_\lambda(g_1)} U(v)^{\oplus m_1}\\
&& \xrightarrow{j_\lambda(f_2)}\cdots \xrightarrow{i_\lambda(g_{t-1})} U(v)^{\oplus m_{t-1}} \xrightarrow{ j_\lambda(f_{t})} j_\lambda(X_t)=j_\lambda(Y),
\end{eqnarray*}
with $\Im (f_i)\subseteq \rad P(a)^{\oplus n_i}$ for any $0\leq i\leq t-1$ and $\Im (g_j)\subseteq \rad Q(b)^{\oplus m_j}$ for any $0\leq j\leq t-1$.
It is similar to Case (c).

To sum up, $KQ^{\Aus}_\Lambda/I^{\Aus}$ is a simple gluing algebra of $KQ^{\Aus}_A/I^{\Aus}_A$ and $KQ^{Aus}_B/I^{\Aus}_B$ by identifying the vertices corresponding to the indecomposable projective modules $P(a)$ and $Q(b)$, and then
$\Aus(\Gp(\Lambda))$ is a simple gluing algebra of $\Aus(\Gp(A))$ and $\Aus(\Gp(B))$ by identifying the vertices corresponding to the indecomposable projective modules $P(a),Q(b)$.
\end{proof}

\begin{example}\label{example Auslander algebra}
Let $\Lambda=KQ/I$ be the algebra where $Q$ is the quiver
as the left quiver in Figure 4 shows,
and $I=\langle \alpha_{i+1}\alpha_i, \varepsilon^2 | i\in\Z/3\Z\rangle$. Let $A$ be the quotient algebra $\Lambda/\langle\varepsilon\rangle$, and $B=e_3\Lambda e_3$.
Then $\Lambda$ is a simple gluing algebra of $A$, $B$. Let $\Aus(\Gp(\Lambda))$ be the algebra corresponding to the bound quiver $(Q^{\Aus},I^{\Aus})$. Then $Q^{\Aus}$
is as the right quiver in Figure 4 shows, and $I^{Aus}=\langle \alpha_{i+1}^+\alpha_i^-, \varepsilon^+\varepsilon^- | i\in\Z/3\Z \rangle $. It is easy to see that
$\Aus(\Gp(\Lambda))$ is the simple gluing algebra of $\Aus(\Gp(A))$ and $\Aus(\Gp(B))$ by identifying the vertices corresponding to the indecomposable projective modules $P(3)$ and $Q(3)$.
\setlength{\unitlength}{0.7mm}
\begin{center}
\begin{picture}(80,40)

\put(0,20){\begin{picture}(50,10)
\put(0,-2){$3$}

\put(2,-2){\vector(1,-1){10}}

\put(12,-14){\vector(-1,0){22}}

\put(-10,-12){\vector(1,1){10}}
\put(13,-15){$2$}
\put(-13,-15){$1$}

\put(7,-6){\tiny$\alpha_2$}
\put(-10,-6){\tiny$\alpha_1$}

\put(-2,-13){\tiny$\alpha_3$}

\qbezier(-1,1)(-3,3)(-2,5.5)
\qbezier(-2,5.5)(1,9)(4,5.5)
\qbezier(4,5.5)(5,3)(3,1)
\put(3.1,1.4){\vector(-1,-1){0.3}}

\put(1,10){$\varepsilon$}

\end{picture}}

\setlength{\unitlength}{0.8mm}
\put(40,-5){\begin{picture}(100,40)

\put(19,11){\vector(-1,1){8}}

\put(34,10){\vector(-1,0){13}}

\put(44,19){\vector(-1,-1){8}}

\put(36,29){\vector(1,-1){8}}

\put(21,30){\vector(1,0){13}}

\put(11,21){\vector(1,1){8}}

\put(8,19){\small$6$}
\put(15,16){\tiny$\alpha_3^+$}
\put(19,8){\small$2$}
\put(34,8){\small$5$}
\put(25,11.5){\tiny$\alpha_2^-$}
\put(35,15){\tiny$\alpha_2^+$}
\put(45,18){\small$3$}

\put(18.9,29){\small$1$}

\put(34,29){\small$4$}
\put(35,22.5){\tiny$\alpha_1^-$}
\put(25,26.5){\tiny$\alpha_1^+$}
\put(15,23){\tiny$\alpha_3^-$}

\put(57,18){\vector(-1,0){9}}

\put(48,21){\vector(1,0){9}}

\put(58,18){\small $7$}
\put(52,22){\tiny$\varepsilon^+$}
\put(51,15){\tiny$\varepsilon^-$}

\end{picture}}

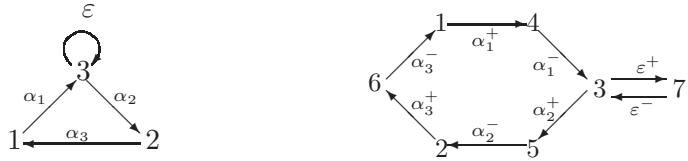
\put(-35,-5){Figure 4. The quiver of $\Lambda$ and its Cohen-Macaulay Auslander algebra.}
\end{picture}
\vspace{0.5cm}
\end{center}

\end{example}
\begin{example}
Following Example \ref{example not gentle algebra}, let $\Lambda=KQ/I$ be the algebra where $Q$ is the quiver $\xymatrix{1 \ar@/^/[r]^{\alpha_1} & 2 \ar@/^/[l]^{\beta_1} \ar@/^/[r]^{\alpha_2} & 3\ar@/^/[l]^{\beta_2}  &4,\ar[l]_\gamma   }$
and $I=\langle \beta_i\alpha_i,\alpha_i\beta_i, \beta_2\gamma |i=1,2 \rangle$. Let $\Aus(\Gp(\Lambda))$ be the algebra corresponding to the bound quiver $(Q^{\Aus},I^{\Aus})$. Then $Q^{\Aus}$
is as Figure 5 shows, and $I^{Aus}=\langle \alpha_{1}^+\beta_1^-,\beta_1^-\alpha_1^+,\alpha_2\beta_2,\beta_2\alpha_2, \beta_2\gamma\rangle $. Let $A$ and $B$ be as in Example \ref{example not gentle algebra}. It is easy to see that
$\Aus(\Gp(\Lambda))$ is the simple gluing algebra of $\Aus(\Gp(A))$ and $\Aus(\Gp(B))$ by identifying the vertices corresponding to the indecomposable projective modules $P(2)$ and $Q(2)$.
\setlength{\unitlength}{0.8mm}
\begin{center}
\begin{picture}(80,40)

\put(10,19.5){\small$1$}

\put(12,22){\vector(1,1){8}}

\put(20.5,31){\small$5$}

\put(23,29.5){\vector(1,-1){8}}

\put(32,19.5){\small$2$}

\put(31,19){\vector(-1,-1){8}}

\put(20.5,9){\small $6$}

\put(19.5,11){\vector(-1,1){8}}

\qbezier(35,22)(42.5,26)(50,22)
\put(50,22){\vector(3,-1){0.2}}

\qbezier(35,20)(42.5,16)(50,20)
\put(35,20){\vector(-3,1){0.2}}
\put(50.5,19.5){\small $3$}

\put(70,19.5){\small $4$}
\put(69,21){\vector(-1,0){16}}

\put(11.5,12){\tiny$\beta_1^-$}
\put(11.5,26){\tiny$\alpha_1^+$}
\put(26.5,12){\tiny$\beta_1^+$}
\put(26.5,26){\tiny$\alpha_1^-$}

\put(40,26){\tiny$\alpha_2$}
\put(40,14){\tiny$\beta_2$}
\put(60,22){\tiny$\gamma$}

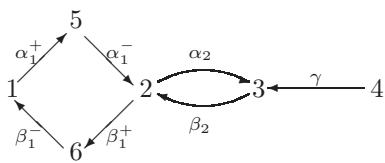
\put(-5,-5){Figure 5. Cohen-Macaulay Auslander algebra of $\Lambda$.}
\end{picture}
\vspace{0.5cm}
\end{center}

\end{example}

\end{document}